\documentclass[US-letter,USenglish,authorcolumns,autoref]{lipics-v2019}

\usepackage[noend]{algpseudocode}
\usepackage{algorithm}
\usepackage{times}
\usepackage{amssymb}
\usepackage{amsmath}
\usepackage{latexsym}
\usepackage{graphics}
\usepackage{url}
\usepackage{verbatim}
\usepackage{color}
\usepackage{setspace}
\usepackage{multirow}
\usepackage{lineno}
\usepackage{booktabs}
\usepackage{graphicx}
\usepackage{caption}
\usepackage[shortlabels]{enumitem}

\definecolor{winered}{rgb}{0.5,0.2,0}

\usepackage{hyperref}
\hypersetup{ 
    colorlinks = true,
    allcolors={winered},
}

\usepackage{wrapfig}
\usepackage[normalem]{ulem}
\usepackage{changepage}

\algnewcommand\And{\textbf{and}}
\algnewcommand\Or{\textbf{or}}
\algnewcommand\Not{\textbf{not}}
\algnewcommand\In{\textbf{in}}
\algnewcommand\Each{\textbf{each}}

\newcommand{\squishlist}{
 \begin{list}{$\bullet$}
  { \setlength{\itemsep}{0pt}
     \setlength{\parsep}{3pt}
     \setlength{\topsep}{3pt}
     \setlength{\partopsep}{0pt}
     \setlength{\leftmargin}{2.5em}
     \setlength{\labelwidth}{1em}
     \setlength{\labelsep}{0.5em} } }

\newcommand{\squishlisttwo}{
 \begin{list}{$\triangleright$}
  { \setlength{\itemsep}{0pt}
     \setlength{\parsep}{0pt}
    \setlength{\topsep}{0pt}
    \setlength{\partopsep}{0pt}
    \setlength{\leftmargin}{2em}
    \setlength{\labelwidth}{1.5em}
    \setlength{\labelsep}{0.5em} } }

\newcommand{\squishend}{
  \end{list}  }

\usepackage{listings}

\definecolor{verbgray}{gray}{0.9}

\lstnewenvironment{code}{%
  \lstset{backgroundcolor=\color{verbgray},
 frame=single,
  language=C,
  framerule=0pt,
  basicstyle=\ttfamily,
  showstringspaces=false,
  commentstyle=\color{blue}\textit,
  keywordstyle=\color{black}\bf,
  numbers=none,
  keepspaces=true,
  columns=fullflexible}}{}

\definecolor{shadecolor}{rgb}{.91, .91, .91}
\definecolor{bordercolor}{rgb}{.8, .8, .6}

\definecolor{ultramarine}{rgb}{0, 0.125, 0.376}

 \definecolor{arsenic}{rgb}{0.23, 0.27, 0.29}
 \definecolor{beige}{rgb}{0.96, 0.96, 0.86}

\definecolor{amber}{rgb}{1.0, 0.75, 0.0}
\definecolor{orange}{rgb}{1.0, 0.49, 0.0}
\definecolor{dandelion}{rgb}{0.94, 0.88, 0.19}

  \definecolor{indiagreen}{rgb}{0.07, 0.53, 0.03}
  \definecolor{huntergreen}{rgb}{0.21, 0.37, 0.23}

\newcommand{\blue}[1] {\textcolor{blue}{#1}}
\newcommand{\red}[1] {\textcolor{red}{#1}}

\newcommand{\bblue}[1] {{\bf \textcolor{blue}{#1}}}
\newcommand{\rred}[1] {{\bf \textcolor{red}{#1}}}

\newcommand{\defo}[1] {\emph{\textcolor{blue}{#1}}}

\definecolor{shadecolor}{rgb}{.9, .9, .9}

 \usepackage{framed}

    {\endMakeFramed}

    \newenvironment{frshaded*}{%
    \MakeFramed {\advance\hsize-\width \FrameRestore}}%
    {\endMakeFramed}
    
    \newcounter{examplecounter}
\newenvironment{exam}{
 \begin{frshaded*}
    \refstepcounter{examplecounter}%
    \noindent
  \textbf{Example \arabic{examplecounter}}%
  \quad
}{%
\end{frshaded*}
}

{\endMakeFramed}

\newenvironment{frshaded2*}{%
    \MakeFramed {\advance\hsize-\width \FrameRestore}}%
    {\endMakeFramed}

\newenvironment{result}{
 \begin{frshaded2*}
}{%
\end{frshaded2*}

}

{\endMakeFramed}

\newenvironment{frshaded3*}{%
    \MakeFramed {\advance\hsize-\width \FrameRestore}}%
    {\endMakeFramed}

\graphicspath{{graphics/}}

\definecolor{winered}{rgb}{0.5,0.2,0}

\usepackage{hyperref}
\hypersetup{ 
    colorlinks = true,
    allcolors={winered},
}

\usepackage{wrapfig}
\usepackage[normalem]{ulem}
\usepackage{changepage}

\DeclareMathOperator{\RCL}{RCL}

\DeclareMathOperator{\ap}{ap}
\DeclareMathOperator{\concat}{concat}
\DeclareMathOperator{\pcrop}{pcr}
\DeclareMathOperator{\parent}{par}
\DeclareMathOperator{\inv}{inv}

\DeclareMathOperator{\convert}{concat}
\DeclareMathOperator{\firstg}{first}

\DeclareMathOperator{\lex}{lex}
\DeclareMathOperator{\colex}{colex}
\DeclareMathOperator{\relex}{relex}
\DeclareMathOperator{\rotateprefix}{rotSet}
\DeclareMathOperator{\rotateSet}{rotprefix}

\newcommand{\bigtree}[1] {\mathcal{T}_{#1}}
\newcommand{\tree} {\mathcal{T}}
\newcommand{\pcr}[1] {\pcrop_{#1}}

\newcommand{\cycletree} {\mathbb{T}}
\newcommand{\subtree} {T}
\renewcommand{\tt} {\mathtt}
\newcommand{\fdown}{{\downarrow}f_1}
\newcommand{\fup}{{\uparrow}f_1}

\newcommand{\A}{\mathbf{A}}
\DeclareMathOperator{\firstone}{first1}
\DeclareMathOperator{\lastone}{last1}

\DeclareMathOperator{\lastzero}{last0}

\newcommand{\rot}[1]{[#1]}

\nolinenumbers

\title{~\\Concatenation trees: A framework for efficient universal cycle and de Bruijn sequence constructions}

\titlerunning{Concatenation Trees}

\author{Joe Sawada}{University of Guelph, Canada}{}{}{}
\author{Jackson Sears}{Ontario Tech University, Canada}{}{}{}
\author{Andrew Trautrim}{University of Guelph, Canada}{}{}{}
\author{Aaron Williams}{Williams College, USA}{}{}{}
\authorrunning{J. Sawada, A. Trautrim, and A. Williams} 
\author{~}{~}{}{}{}
\authorrunning{~}

\Copyright{~}

\begin{document}
\pagenumbering{gobble}
\maketitle

\begin{abstract}

Classic cycle-joining techniques have found widespread application in creating universal cycles for a diverse range of combinatorial objects, such as shorthand permutations, weak orders, orientable sequences, and various subsets of $k$-ary strings, including de Bruijn sequences. In the most favorable scenarios, these algorithms operate with a space complexity of $O(n)$ and require $O(n)$ time to generate each symbol in the sequences.  In contrast, concatenation-based methods have been developed for a limited selection of universal cycles. In each of these instances, the universal cycles can be generated far more efficiently, with an amortized time complexity of $O(1)$ per symbol, while still using $O(n)$ space.
This paper introduces \emph{concatenation trees}, which serve as the fundamental structures needed to bridge the gap between cycle-joining constructions 
and corresponding concatenation-based approaches. They immediately demystify the relationship between the classic Lyndon word (necklace) concatenation construction of de Bruijn sequences and a corresponding cycle-joining based construction. To underscore their significance, concatenation trees are applied to construct universal cycles for shorthand permutations, weak orders, and orientable sequences in $O(1)$-amortized time per symbol.  

\end{abstract}

\bigskip

\noindent

\newpage
\pagenumbering{arabic}
\section{Introduction} \label{sec:intro}

Readers are likely familiar with the concept of a \defo{de Bruijn sequence} (DB sequence), which is a circular string of length $k^n$ in which every $k$-ary string of length $n$ appears once as a substring.
For example, a binary DB sequence for $n=4$ is $0000100110101111$.
%
%
The study of these sequences dates back to Pingala's \emph{Chanda\d{h}\'{s}\=astra} \includegraphics[height=0.7em]{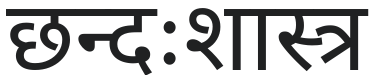} (`A Treatise on Prosody') over two thousand years ago (see \cite{kak2000yamatarajabhanasalagam,stein1961mathematician,stein2013mathematics,van1993binary}).
They have a wide variety of well-known modern-day applications~\cite{dbseq} and their theory is even being applied to de novo assembly of read sequences into a genome~\cite{nano2,compeau2011apply,nano3,nano1,genome}.
More broadly, when the underlying objects are not $k$-ary strings, the analogous concept is often called a \emph{universal cycle} \cite{chung1992universal}, and they have been studied for many  fundamental objects including permutations~\cite{shorthand2,johnson,shorthand,shorthandWong}, combinations~\cite{curtis,hurlbert,jackson}, set partitions~\cite{higgins}, and  graphs~\cite{brockman}.

In this paper, we develop a concatenation framework for the generation of DB sequences and universal cycles. We prove that each such sequence is equivalent to one generated by a corresponding successor rule that is based on an underlying cycle-joining tree. As we demonstrate, the concatenation constructions can often be implemented to generate the sequences in $O(1)$-amortized time per symbol, whereas the corresponding successor-rule generally requires $O(n)$ time. 
To illustrate our approach, it is helpful to consider a slightly more complex object.
A \defo{weak order} is a way competitors can rank in an event, where ties are allowed.
For example, in a horse race with five horses labeled $h_1, h_2, h_3, h_4, h_5$, the weak order (using a rank representation) 22451 indicates $h_5$ finished first, the horses $h_1$ and $h_2$ tied for second, horse $h_3$ finished fourth, and horse $h_4$ finished fifth.  
No horse finished third as a result of the tie for second.
Let $\mathbf{W}(n)$ denote the set of weak orders of order $n$.
For example, the thirteen weak orders for $n=3$ are given below:
\begin{equation*} \label{eq:weak3}
\mathbf{W}(3) = \{111,113,131,311,122,212,221,123,132,213,231,312,321\}.
\end{equation*}
\noindent
Note that $\mathbf{W}(n)$ is closed under rotation. 
For this reason, we can apply the pure cycling register (PCR), which corresponds to the function  $f(\tt{a}_1\tt{a}_2\cdots \tt{a}_n) = \tt{a_2}\cdots\tt{a}_n\tt{a_1}$, to induce small cycles.  Then, we repeatedly join the smaller cycles together to obtain a universal cycle.
In this approach, we partition $\mathbf{W}(n)$ into equivalence classes under rotation.
These classes are called \defo{necklaces} and we use the lexicographically smallest member of each class as its representative.
So $\{113, 131, 311\}$ is one class with representative $113$, and $\{111\}$ is another class.
Each class of size $t$ has a universal cycle of length $t$, namely the representative's aperiodic prefix (i.e., the shortest prefix of a string that can be concatenated some number of times to create the entire string).
So $113$ is a universal cycle for $\{113, 131, 311\}$, and $1$ is a universal cycle for $\{111\}$ (since $1$ is viewed cyclically).
Each necklace class can be viewed as a directed cycle induced by the PCR, where each edge corresponds to a rotation (i.e., the leftmost symbol is shifted out and then shifted back in as the new rightmost symbol), as seen in Figure~\ref{fig:weak3_necklaces} for $n=3$.  Two cycles can be joined together via a conjugate pair (formally defined in Section~\ref{sec:cycle-join}) to create a larger cycle as illustrated in Figure~\ref{fig:weak3_join}.   This is done by replacing a pair of rotation edges with a pair of edges that shift in a new symbol.  Repeating this process yields a universal cycle $1113213122123$ for $\mathbf{W}(3)$.

\begin{figure}[h]
    \begin{subfigure}{0.55\textwidth}
        \centering
        \includegraphics[scale=0.75]{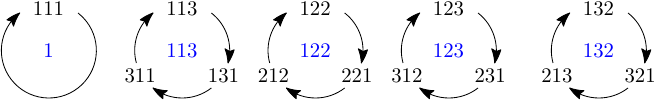} \ \ \ 
        \caption{Necklace cycles for $\mathbf{W}(3)$, where the representative of each class is at the top of its cycle, and the universal cycle is in the middle.}
        \label{fig:weak3_necklaces}
    \end{subfigure}
    \hfill
    \begin{subfigure}{0.43\textwidth}
        \centering
        \includegraphics[scale=0.75]{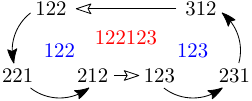} 
        \caption{Necklace cycles $\blue{122}$ and $\blue{123}$ are joined into a single cycle.
        The universal cycle for these strings is $\red{122123}$.}
        \label{fig:weak3_join}
    \end{subfigure}
        \caption{Initial steps to building a universal cycle for $\mathbf{W}_{3}$.}

    \label{fig:weak3}
\end{figure}
%
\begin{figure}[ht] 
    \centering
    \begin{subfigure}{0.49\textwidth}
        \centering
        \includegraphics[scale=0.79]{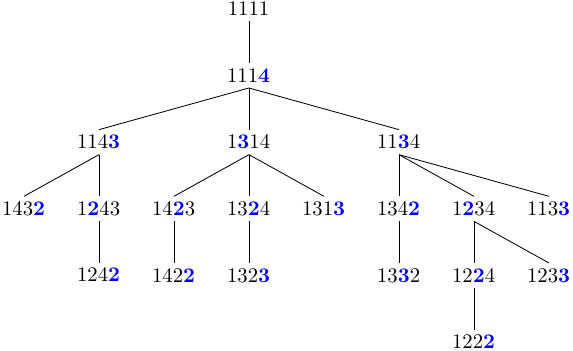}
        \caption{A cycle-joining tree for weak orders when $n=4$. 
        The precise parent rule appears in Section \ref{sec:weak}.}
        \label{fig:weakTrees4_cycleJoin}        
    \end{subfigure}
    \hfill
    \begin{subfigure}{0.49\textwidth}
        \centering
        \includegraphics[scale=0.79]{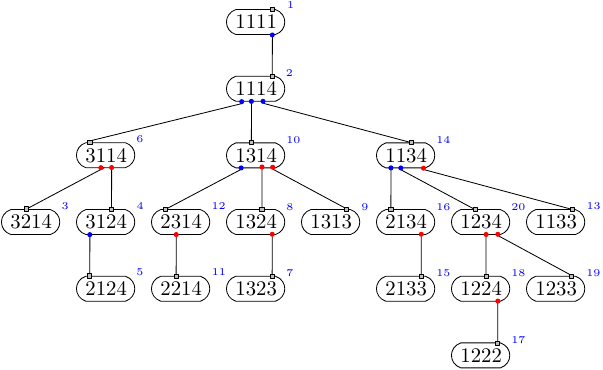}
        \caption{A concatenation tree $\bigtree{weak}$ for weak orders when $n=4$ illustrating the RCL order.   
        }
        \label{fig:weakTrees4_concatenation}
    \end{subfigure}
    \caption{Two tree structures for creating a universal cycle for $\mathbf{W}_{4}$.
  }
\label{fig:weak}
\end{figure}

In many cases, pairs of cycles can be joined together to form a cycle-joining tree.
For example, Figure~\ref{fig:weakTrees4_cycleJoin} illustrates a cycle-joining tree for $\mathbf{W}(4)$ based on an explicit parent rule stated in Section~\ref{sec:weak}.  Given a cycle-joining tree,
existing results in the literature~\cite{binframework,karyframework} allow us to generate a corresponding universal cycle \emph{one symbol at a time}. 
But what if we want to generate the universal cycle faster?
For instance, suppose that instead of generating one symbol at a time, we can generate necklaces one at a time.\footnote{In practice, a DB sequence (or universal cycle) does not need to be returned to an application one symbol at a time, but rather
a word can be shared between the generation algorithm and the application.
The algorithm repeatedly informs the application that the next batch of symbols in the sequence is ready.
This allows the generation algorithm to slightly modify the shared word and provide $O(n)$ symbols to the application as efficiently as  $O(1)$-amortized time~\cite{grandma2}.}  
How can we do this?
This  goal of generating one necklace at a time has been achieved in only a handful of  cases~\cite{grandma2,lex-comp,GS18,cool-lex,multi}.
Most notably, the DB sequence known as the Ford sequence, or the  \emph{Granddaddy} (see Knuth \cite{knuth}), can be created by concatenating the associated representatives in lexicographic order~\cite{fkm2}, matching the DB sequence given earlier:  $0~0001~0011~01~0111~1$.  
But these concatenation constructions have been the exception rather than the rule, and there has been no theoretical framework for understanding why they work.
Here, we provide the missing link.
For example, the unordered cycle joining tree in Figure \ref{fig:weakTrees4_cycleJoin} is redrawn in Figure \ref{fig:weakTrees4_concatenation}.
The new diagram is a bifurcated ordered tree (formally defined in Section~\ref{sec:bots}), meaning that children are ordered and partitioned into left and right classes, and importantly some representatives have changed.
If the tree is explored using an \emph{RCL traversal} (i.e., right children, then current, then left children), then --- presto! --- a concatenation construction of a universal cycle for $\mathbf{W}(4)$ is created:
\begin{equation*} \label{eq:weakUcycle4_concat}
1~1114~3214~3124~2124~3114~1323~1324~13~1314~2214~2314~1133~1134~2133~2134~1222~1224~1233~1234.
\end{equation*}

\begin{result}
\noindent
 {\bf Main result: } This paper introduces  \emph{concatenation trees} and \emph{RCL traversals}, which bridge the gap between $k$-ary PCR-based cycle-joining trees and concatenation constructions for corresponding universal cycles.  We apply the framework to construct universal cycles in $O(1)$-amortized time per symbol using polynomial space for (1) shorthand permutations, (2) weak orders, (3) orientable sequences, and (4) DB sequences.

\end{result}

\noindent
Our main result generalizes many interesting results for DB sequences and their relatives, with details provided in Section~\ref{sec:big4} and Section~\ref{sec:dbseq}.  
\begin{enumerate}
    \item It demystifies the relationship between the successor rule and the concatenation construction of the previously mentioned Granddaddy DB sequence~\cite{fred-succ,fkm2}, by providing a clear correspondence between the concatenation construction and the successor rule derived from an underlying cycle-joining tree.

    \item Similar to the Grandaddy, it demystifies the relationship between the known successor rule and concatenation construction of the Grandmama DB sequence~\cite{grandma2}.
    
    \item It provides the first proof of an observed correspondence between a successor rule construction~\cite{jansen,wong}  and a simple concatenation construction observed in~\cite{GS18} (that we later name the \emph{Granny} DB sequence).

    \item It generalizes known results for bounded weight universal cycles~\cite{walcom,min-weight,generalized-greedy} and universal cycles with forbidden $0^j$ substring ~\cite{GS18,generalized-greedy}; the latter has recent application in quantum key distribution schemes~\cite{cheeRLL}.
\end{enumerate}

\noindent
Additionally, we apply the framework to other combinatorial objects to highlight its general significance. 
\begin{enumerate}
\item Concatenation trees are applied to a $O(n)$ time per symbol
cycle-joining construction for shorthand permutations~\cite{karyframework} to generate the same universal cycle in $O(1)$-amortized time per symbol using $O(n^2)$ space.

\item Concatenation trees are applied to a $O(n)$ time per symbol
cycle-joining construction for weak orders~\cite{weakorder} to generate the same universal cycle in $O(1)$-amortized time per symbol using $O(n^2)$ space.

\item Concatenation trees are applied to a  $O(n)$ time per symbol
cycle-joining construction for orientable sequences~\cite{G&S-Orientable:2024} to generate the same universal cycle in $O(1)$-amortized time per symbol using $O(n^2)$ space.
\end{enumerate}

\noindent


While our focus is on PCR-based cycle-joining trees, preliminary evidence indicates that our framework can be generalized (though non-trivially) to other underlying feedback functions in the binary case including:
\begin{itemize}
    \item the Complementing Cycling Register (CCR) with feedback function $f(\tt{a}_1\tt{a}_2\cdots \tt{a}_n) = 1 \oplus \tt{a}_1 =  \overline{\tt{a}}_1$,
     \item the Pure Summing Register (PSR) with feedback function $f(\tt{a}_1\tt{a}_2\cdots \tt{a}_n) = \tt{a}_1 \oplus \tt{a}_2 \cdots \oplus \tt{a}_n$,
          \item the Complementing Summing Register (CSR) with feedback function $f(\tt{a}_1\tt{a}_2\cdots \tt{a}_n) = 1 \oplus \tt{a}_1 \oplus \tt{a}_2 \cdots \oplus \tt{a}_n$, and
    \item the Pure Run-length Register (PRR) with feedback function $f(\tt{a}_1\tt{a}_2\cdots \tt{a}_n) = \tt{a}_1 \oplus \tt{a}_2 \oplus \tt{a}_n$,
\end{itemize}

\noindent
where $\oplus$ is addition modulo 2, and $\overline{\tt{x}}$ is the complement of $\tt{x}$.
This has the potential to unify a large body of independent results, enabling new and interesting results.  
In particular, the recently introduced pure run-length register (PRR)~\cite{pref-same} is conjectured to be the underlying feedback function used in a lexicographic composition construction~\cite{lex-comp}. 
Furthermore, the PRR is proved to be the underlying function used in the greedy prefer-same~\cite{eldert} and prefer-opposite~\cite{pref-opposite}  constructions; however, no concatenation construction is known.
The first successor rule based on the complementing cycling register (CCR) is noted to have a very good local 0-1 balance~\cite{huang}; however, no corresponding concatenation construction is known. 
There are two known CCR-based concatenation constructions~\cite{Gabric2017,GS18}, but there is no clear correlation to an underlying cycle-joining approach, even though one appears to be equivalent to a successor rule from~\cite{binframework}.
The cool-lex concatenation constructions~\cite{cool-lex} have equivalent underlying successor rules based on the pure summing register (PSR) and the  complementing summing register (CSR).  This correspondence was not observed until considering larger alphabets~\cite{multi}, though little insight to the correspondence is provided in the proof.  Cycle-joining constructions based on the PSR/CSR are also considered in~\cite{Etzion1987,Etzion1984}.

\smallskip

\noindent
{\bf Outline.}
In Section~\ref{sec:back}, we present  the necessary background definitions and notation along with a detailed discussion of cycle-joining trees and their corresponding successor rules.  In Section~\ref{sec:bots}, we introduce bifurcated ordered trees, which are the structure underlying concatenation trees. In Section~\ref{sec:concat}, we introduce concatenation trees along with a statement of our main result.  In Section~\ref{sec:app}, we apply our framework to a wide variety of interesting combinatorial objects, including DB sequences.  
Implementation of the universal cycle algorithms presented in this paper are available at \url{http://debruijnsequence.org}.

\section{Preliminaries}  \label{sec:back}

Let $\Sigma = \{0,1,2, \ldots , k-1\}$ denote an alphabet with $k$ symbols.  
Let $\Sigma^n$ denote the set of all length-$n$ strings over $\Sigma$.  
Let $\alpha = \tt{a}_1\tt{a}_2\cdots \tt{a}_n$ denote a string in $\Sigma^n$.  
The notation $\alpha^t$ denotes $t$ copies of $\alpha$ concatenated together.
The \defo{aperiodic prefix} of $\alpha$ 
is the shortest string $\beta$ such that $\alpha = \beta^t$ for some $t\geq 1$; the \defo{period} of $\alpha$ is $|\beta|$.  Let $\ap(\alpha_1,\alpha_2, \ldots , \alpha_n)$ denote the concatenation of the aperiodic prefixes of $\alpha_1, \alpha_2, \ldots, \alpha_t$.  For example $\ap(0000, 0111, 1010) = 0011110$, and $\ap(010101) = 01$. Note that $010101$ has period equal to 2.
If the period of $\alpha$ is $n$, then $\alpha$ is said to be \defo{aperiodic} (or primitive); otherwise, it is said to be \defo{periodic} (or a proper power).  When $k=2$, let $\overline{\tt{a}}_i$ denote the complement of a bit $\tt{a}_i$.  


A \defo{necklace class} is an equivalence class of strings under rotation.  
A \defo{necklace} is the lexicographically smallest representative of a necklace class.  A \defo{Lyndon word} is an aperiodic necklace. Let $\mathbf{N}_k(n)$ denote the set of all $k$-ary necklaces of order $n$.  As an example, 
the six binary necklaces for $n=4$ are: $\mathbf{N}_2(4) = \{ 0000, 0001, 0011, 0101, 0111, 1111\}$.  
Let $\rot{\alpha}$ denote the set of all strings in $\alpha$'s necklace class, i.e., the set of all rotations of $\alpha$.  For example,
$\rot{0001} = \rot{1000} = \{0001, 0010, 0100, 1000\}$ and $\rot{0101} = \{0101, 1010\}$.
The \defo{pure cycling register} (PCR) is a shift register with feedback function $f(\tt{a}_1\tt{a}_2\cdots \tt{a}_n) = \tt{a}_1$. Starting with $\alpha$, it induces a cycle containing the strings in $\alpha$'s necklace class.  For example, 
$$0001 \rightarrow 0010 \rightarrow 0100 \rightarrow 1000 \rightarrow  \mathit{0001}$$ 
is a cycle induced by the PCR that can be represented by any string in the cycle.
Given a tree $\subtree$ with nodes (cycles induced by the PCR) labeled by necklace representatives $\{\alpha_1, \alpha_2, \ldots, \alpha_t\}$, let $\mathbf{S}_{\subtree} = \rot{\alpha_1} \cup \rot{\alpha_2} \cup \cdots \cup \rot{\alpha_t}$.
For example, if $n=4$ and  $T$ contains two nodes $\{0001, 0101\}$ then 
$\mathbf{S}_{T} = \{0001, 0010, 0100, 1000\} \ \cup \ \{0101, 1010\}$.

%

%
%
%

Given $\mathbf{S} \subseteq \Sigma^n$, a \defo{universal cycle} $U$ for $\mathbf{S}$ is a cyclic sequence of length $|\mathbf{S}|$ that contains each string in $\mathbf{S}$ as a substring (exactly once).   
Given a universal cycle $U$ for a set $\mathbf{S} \subseteq \Sigma^n$, a \defo{successor rule} for $U$ is a function $f:\mathbf{S} \rightarrow \Sigma$ such that $f(\alpha)$ is the symbol following $\alpha$ in $U$.  


\subsection{Cycle joining trees}  \label{sec:cycle-join}

In this section we review how two universal cycles can be joined  to obtain a larger universal cycle. 
Let $\tt{x},\tt{y}$ be distinct symbols in $\Sigma$.
If $\alpha = \tt{x}\tt{a}_2\cdots \tt{a}_n$ and $\hat \alpha = \tt{y}\tt{a}_2\cdots \tt{a}_n$, then $\alpha$ and $\hat \alpha$ are said to be \defo{conjugates} of each other, and $(\alpha, \hat \alpha)$ is called a \defo{conjugate pair}.
The following well-known result (see for instance Lemma 3 in~\cite{dbrange}) based on conjugate pairs is the crux of the cycle-joining approach.\footnote{The cycle-joining approach has graph theoretic underpinnings related to Hierholzer's algorithm for constructing Euler cycles~\cite{hierholzer}.}  

\begin{theorem} \label{thm:concat}
Let $\mathbf{S}_1$ and $\mathbf{S}_2$ be disjoint subsets of $\Sigma^n$
such that $\alpha = \tt{x}\tt{a}_2\cdots \tt{a}_n \in \mathbf{S}_1$ and $\hat \alpha = \tt{y}\tt{a}_2\cdots \tt{a}_n \in \mathbf{S}_2$; $(\alpha, \hat \alpha)$ is a conjugate pair. 
If $U_1$  is a universal cycle for $\mathbf{S}_1$ with suffix $\alpha$ and $U_2$ is a universal cycle for $\mathbf{S}_2$  with suffix $\hat \alpha$ then $U = U_1U_2$ is a universal cycle for $\mathbf{S}_1 \cup \mathbf{S}_2$.
\end{theorem}

Let $U_i$ denote a universal cycle for $\mathbf{S}_i \subseteq \Sigma^n$. Two universal cycles $U_1$ and $U_2$ are said to be \defo{disjoint} if $\mathbf{S}_1 \cap \mathbf{S}_2 = \emptyset$.
A \defo{cycle-joining tree} $\cycletree$  is an unordered tree where 
the nodes correspond to a disjoint set of universal cycles $U_1, U_2, \ldots, U_t$; an edge between $U_i$ and $U_j$ is defined by a conjugate pair $(\alpha, \hat \alpha)$ such that $\alpha \in \mathbf{S}_i$ and $\hat \alpha \in \mathbf{S}_j$. 
For our purposes, we consider cycle-joining trees to be rooted.  
If the cycles are induced by the PCR, i.e., the cycles correspond to necklace classes, then $\cycletree$ is said to be a \defo{PCR-based cycle-joining tree}.  
As examples,  four PCR-based cycle-joining trees are illustrated in Figure~\ref{fig:big4}; their nodes are labeled by the necklaces $\mathbf{N}_2(6)$.  They are defined by the following  \defo{parent-rules}, which determines the parent of a given non-root node.   

\begin{result}
\noindent
{\bf Four ``simple'' parent rules defining binary PCR-based cycle-joining trees}

\begin{itemize}
\item $\cycletree_1$: rooted at $1^n$ and the parent of every other node $\alpha \in \mathbf{N}_2(n)$ is obtained by flipping the \blue{last 0}.
\item $\cycletree_2$: rooted at $0^n$ and the parent of every other node $\alpha \in \mathbf{N}_2(n)$ is obtained by flipping the \blue{first 1}.
\item $\cycletree_3$: rooted at $0^n$ and the parent of every other node $\alpha \in \mathbf{N}_2(n)$ is obtained by flipping the \blue{last 1}.
\item $\cycletree_4$:  rooted at $1^n$ and the parent of every other node $\alpha \in \mathbf{N}_2(n)$ is obtained by flipping the \blue{first 0}.
\end{itemize}

\vspace{-0.15in}

\end{result}

\noindent
Note that for $\cycletree_3$  and $\cycletree_4$, the parent of a node $\alpha$ is obtained by first flipping the named bit and then rotating the string to its lexicographically least rotation to obtain a necklace.
%
Each node $\alpha$ and its parent $\beta$ are joined by a conjugate pair where the highlighted bit in $\alpha$ is the first bit in one of the conjugates.  For example, the nodes $\alpha = 0\rred{1}1011$ and $\beta = 001011$ in $\cycletree_2$ from Figure~\ref{fig:big4} are joined by the conjugate pair $(1\underline{10110}, 0\underline{10110})$.

%
\begin{figure}[ht]
\centering
\resizebox{6.2in}{!}{\includegraphics{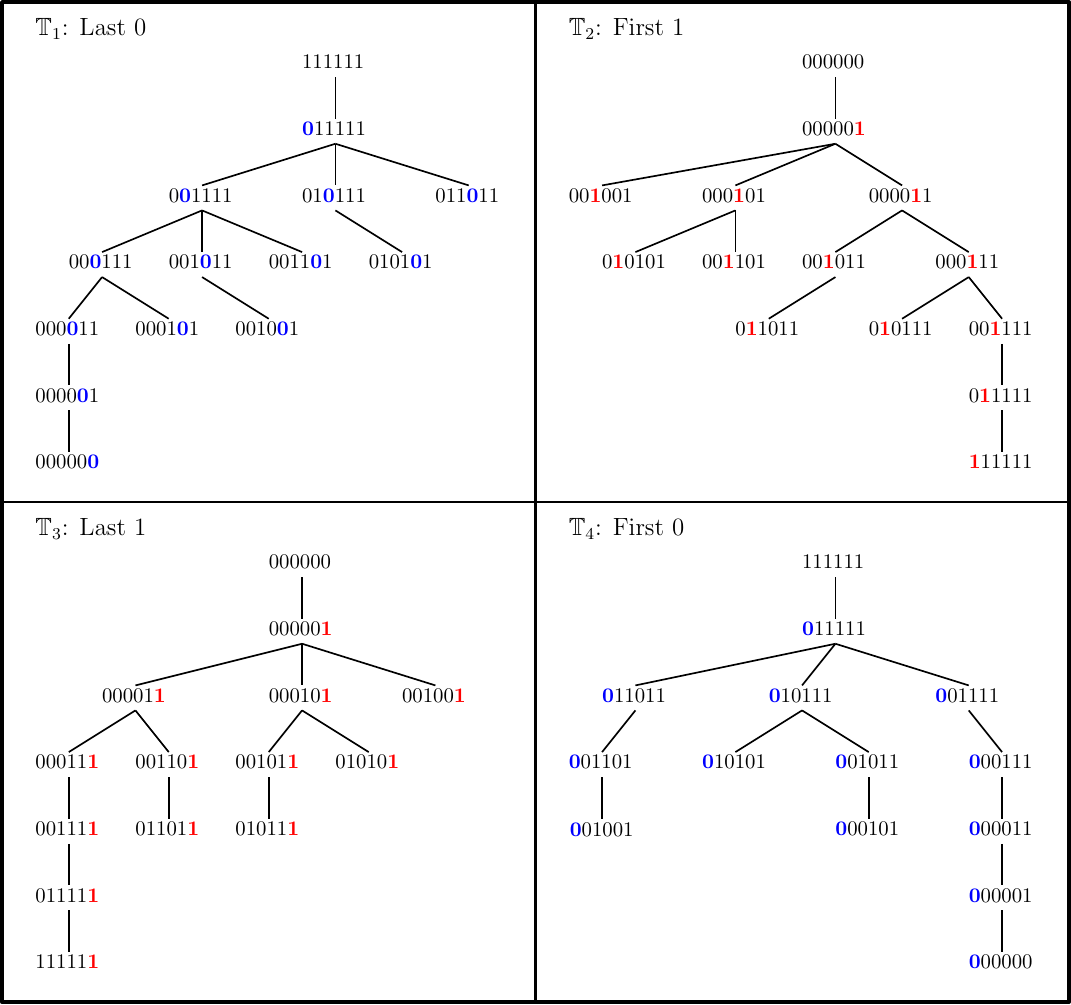}}
\caption{
Cycle-joining trees for $n=6$ and $k=2$ derived from the four simple parent rules.  
The node $001101$ is joined to a different parent cycle in each tree.
In particular, the edge $0011\blue{0}1 \text{--} 001111$ in $\cycletree_1$ is obtained by flipping its \blue{last~$0$}.}  
\label{fig:big4}
\end{figure}

When two adjacent nodes $U_i$ and $U_j$ in a cycle-joining tree $\cycletree$ are joined  to obtain $U$ via Theorem~\ref{thm:concat} (rotating the cycles as appropriate), the nodes are unified and replaced with $U$ (the edge between $U_i$ and $U_j$ is contracted).  Repeating this process until only one node remains produces a universal cycle for $\mathbf{S}_1 \cup \mathbf{S}_2 \cup \cdots \cup \mathbf{S}_t$.  In the binary case, the same universal cycle is produced, no matter the order in which the cycles are joined.  This is because no string can belong to more than one conjugate pair in the underlying definition of $\cycletree$.  However, when $k>2$, the order that the cycles are joined can be important.

\begin{exam} \label{exam:chain} \small
The following illustrates two different ways to join the cycles in a PCR-based cycle-joining tree $\cycletree$ for $n=3$ and $k=3$ with three nodes represented by $001, 002,$ and  $003$ joined via conjugate pairs $(100, 200)$, $(200, 300)$.  Note the string $200$ belongs to both conjugate pairs. 
\begin{center}
\resizebox{4.5in}{!}{\includegraphics{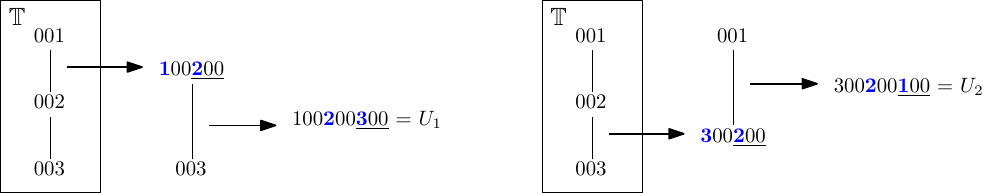}}
\end{center}
\noindent
The resulting universal cycle for $\mathbf{S}_\cycletree = \rot{001} \cup \rot{002} \cup \rot{003}$ is different in each case.

\end{exam}

In upcoming discussion regarding both successor rules and concatenation trees, we require the  underlying cycle-joining trees to have the following property when $k>2$.


\begin{result}
\noindent
{\bf Chain Property}: If a node in a cycle-joining tree $\cycletree$ has two children joined via conjugate pairs $(\tt{x}\tt{a}_2\cdots \tt{a}_n, \tt{y}\tt{a}_2\cdots \tt{a}_n)$ and 
$(\tt{x}'\tt{b}_2\cdots \tt{b}_n, \tt{y'}\tt{b}_2\cdots \tt{b}_n)$, then $\tt{a}_2\cdots \tt{a}_n \neq \tt{b}_2\cdots \tt{b}_n$. 
\end{result}

\begin{wrapfigure}[3]{r}{0.20\textwidth}  
\vspace{-0.2in}
        \centering
       \resizebox{0.75in}{!}{\includegraphics{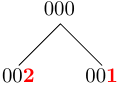}  }     
\end{wrapfigure}
\noindent 
Observe that the Chain Property is satisfied in Example~\ref{exam:chain}, and is always satisfied when $k=2$.  The cycle-joining tree on the right with conjugate pairs ($0\underline{00},1\underline{00}$) and ($0\underline{00},2\underline{00}$) illustrates a rooted tree that does not satisfy the Chain Property.


\subsection{Successor-rule constructions} \label{sec:big4}

Let $\cycletree$ be a PCR-based cycle-joining tree  where the nodes are joined by a set $\mathbf{C}$ of conjugate pairs.  We say $\gamma$ \defo{belongs to} a conjugate pair $(\alpha, \hat \alpha)$ if either $\gamma = \alpha$ or $\gamma = \hat \alpha$.  If $k=2$, the following function $f_0$ is a successor rule for the corresponding universal cycle for $\mathbf{S}_{\cycletree}$~\cite{binframework}, where $\alpha = \tt{a}_1\tt{a}_2\cdots \tt{a}_n$:
\begin{center}
$f_0(\alpha) = \left\{ \begin{array}{ll}
         \overline{\tt{a}}_1 &\ \  \mbox{if $\alpha$ belongs to some conjugate pair in $\mathbf{C}$;}\\
         \tt{a}_1 \  &\ \  \mbox{otherwise.}\end{array} \right. $
\end{center}


Applying the successor rule $f_0$ directly requires an exponential amount of memory to store the conjugate pairs. However, a cycle-joining tree defined by a straightforward parent rule may allow for a much more efficient implementation, using as little as $O(n)$ space and $O(n)$ time.  Recall the four parent rules stated for the trees $\cycletree_1$, $\cycletree_2$, $\cycletree_3$, $\cycletree_4$.
The upcoming four successor rules $\pcr{1}, \pcr{2}, \pcr{3}, \pcr{4}$, which correspond to $f_0$, are stated generally for any subtree $\subtree$ of the corresponding cycle-joining tree; they will be revisited in Section~\ref{sec:dbseq}.
Previously, these successor rules were stated for the entire trees in~\cite{binframework}, and then for subtrees that included all nodes up to a given level~\cite{karyframework}
putting a restriction on the minimum or maximum weight (number of 1s) of any length-$n$ substring. 

\begin{result}  
\noindent   
\underline{$\cycletree_1$ (Last 0)} 
\small
~
\noindent
Let $j$ be the smallest index of $\alpha = \tt{a}_1\tt{a}_2\cdots \tt{a}_n$ such that $\tt{a}_j = 0$ and $j > 1$, or $j=n{+}1$ if no such index exists.
Let $\gamma = \tt{a}_j\tt{a}_{j+1}\cdots \tt{a}_n 0 \tt{a}_2\cdots \tt{a}_{j-1} = \tt{a}_j\tt{a}_{j+1}\cdots \tt{a}_n 0 1^{j-2}$.

\begin{center}
$\pcr{1}(\alpha) = \left\{ \begin{array}{ll}
         \overline{\tt{a}}_1 &\ \  \mbox{if $\gamma$ is a necklace \blue{and $\tt{a}_2\cdots \tt{a}_n\overline{\tt{a}}_1 \in~ \mathbf{S}_{\subtree}$};}\\
         {\tt{a}_1} \  &\ \  \mbox{otherwise.}\end{array} \right. $
\end{center}

\vspace{-0.2in}
%
\end{result}

\vspace{-0.2in}
\begin{result}  
\noindent  \underline{$\cycletree_2$ (First 1)} 
\small
~Let $j$ be the largest index of $\alpha = \tt{a}_1\tt{a}_2\cdots \tt{a}_n$ such that $\tt{a}_j = 1$, or $j=0$ if no such index exists.
Let $\gamma = \tt{a}_{j+1}\tt{a}_{j+2}\cdots \tt{a}_n 1 \tt{a}_2\cdots \tt{a}_j = 0^{n-j}1\tt{a}_2\cdots \tt{a}_j$.

\begin{center}
$\pcr{2}(\alpha) = \left\{ \begin{array}{ll}
         \overline{\tt{a}}_1 &\ \  \mbox{if $\gamma$ is a necklace
         \blue{and $\tt{a}_2\cdots \tt{a}_n\overline{\tt{a}}_1 \in~ \mathbf{S}_{\subtree}$};}\\
         {\tt{a}_1} \  &\ \  \mbox{otherwise.}\end{array} \right. $
\end{center}

\vspace{-0.2in}
\end{result}

\vspace{-0.2in}
\begin{result}  
\noindent  \underline{$\cycletree_3$ (Last 1) }
 \small
~Let $\alpha = \tt{a}_1\tt{a}_2\cdots \tt{a}_n$ and let $\gamma = \tt{a}_2 \tt{a}_3 \cdots \tt{a}_{n}1$.

\begin{center}
$\pcr{3}(\alpha) = \left\{ \begin{array}{ll}
         \overline{\tt{a}}_1 &\ \  \mbox{if $\gamma$ is a necklace
         \blue{and $\tt{a}_2\cdots \tt{a}_n\overline{\tt{a}}_1 \in~ \mathbf{S}_{\subtree}$};}\\
         {\tt{a}_1} \  &\ \  \mbox{otherwise.}\end{array} \right. $
\end{center}

\vspace{-0.2in}
\end{result}

\vspace{-0.2in}
\begin{result}  
\noindent \underline{$\cycletree_4$ (First 0)} 
 \small
~
Let $\alpha = \tt{a}_1\tt{a}_2\cdots \tt{a}_n$ and let $\gamma = 0\tt{a}_2 \tt{a}_3 \cdots \tt{a}_{n}$.

\begin{center}
$\pcr{4}(\alpha) = \left\{ \begin{array}{ll}
         \overline{\tt{a}}_1 &\ \  \mbox{if $\gamma$ is a necklace
         \blue{and $\tt{a}_2\cdots \tt{a}_n\overline{\tt{a}}_1 \in~ \mathbf{S}_{\subtree}$};}\\
         {\tt{a}_1} \  &\ \  \mbox{otherwise.}\end{array} \right. $
\end{center}

\vspace{-0.2in}
\end{result}

The DB sequences obtained by applying the four successor rules for $n=6$ and $k=2$
to $T = \cycletree_1,\cycletree_2,\cycletree_3,\cycletree_4$, respectively, are provided in Table~\ref{tab:db6}.  The spacing between some symbols are used 
to illustrate the correspondence to upcoming concatenation constructions.
%
\begin{table}[h]
\begin{center}
\begin{tabular}{c |l} 
 {\bf Successor rule} & {\bf DB sequence for $n=6$ and $k=2$} \\ \hline
$\pcr{1}$  & $0~000001~000011~000101~000111~001~001011~001101~001111~01~010111~011~011111~1$ \\
$\pcr{2}$  & 
$0~000001~001~000101~01~001101~000011~001011~011~000111~010111~001111~011111~1$
\\
$\pcr{3}$  & 
$1~111110~111100~111000~110~110100~110000~101110~101100~10~101000~100~100000~0$ \\
$\pcr{4}$  &
$1~111110~110~100~100110~111010~10~110010~100010~111100~111000~110000~100000~0$ \\
\end{tabular}

\caption{DB sequences resulting from the successor rules corresponding to the cycle-joining trees $\cycletree_1$, $\cycletree_2$, $\cycletree_3$, $\cycletree_4$ from Figure~\ref{fig:big4}.}
\label{tab:db6}
\end{center}
\end{table}
%
%
The DB sequence generated by $\pcr{1}$ is the well-known Ford sequence~\cite{fred-succ}, and is called the \emph{Granddaddy} by Knuth~\cite{knuth}.  It is the lexicographically smallest DB sequence, and it can also be generated by a prefer-0 greedy approach attributed to Martin~\cite{martin}. Furthermore, Fredricksen and Maiorana~\cite{fkm2} demonstrate an equivalent necklace (or Lyndon word) concatenation construction  that can generate the sequence in $O(1)$-amortized time per bit.   
The DB sequence generated by $\pcr{2}$ 
is called the \emph{Grandmama} by Dragon et al.~\cite{grandma2}; it 
can also be generated in $O(1)$-amortized time per bit by concatenating necklaces in co-lexicographic order. %
The DB sequence generated by $\pcr{3}$, was first discovered by Jansen~\cite{jansen} for $k=2$, then generalized in~\cite{wong}. It is conjectured to have a concatenation construction by Gabric and Sawada~\cite{GS18}, a fact we prove in Section~\ref{sec:dbseq}.
The DB sequence generated by $\pcr{4}$, was first discovered by Gabric, Sawada, Williams, and Wong~\cite{binframework}.  No concatenation construction for this sequence was previously known which served as the initial motivation for this work. 

\subsubsection{Non-binary alphabets} \label{sec:succ-kary}


Consider a non-binary alphabet where $k>2$.
Recall from Example~\ref{exam:chain}, that the order the cycles are joined in a cycle-joining tree $\cycletree$ may be important. This means defining a natural and generic successor rule  is more challenging, especially if $\cycletree$ does not satisfy the Chain Property, i.e., $\cycletree$ has a node with two children joined via conjugate pairs of the form $(\tt{x}\beta, \tt{y}\beta)$ and $(\tt{x}\beta, \tt{z}\beta)$, for some $k$-ary string $\beta$. Thus, going forward, assume that $\cycletree$ satisfies the Chain Property.


Let $\alpha_1, \alpha_2, \ldots, \alpha_m$ denote a maximal length path of nodes in $\cycletree$ such that for each $1 \leq i < m$, the node $\alpha_i$ is the parent of $\alpha_{i+1}$
and they are joined via a conjugate pair of the form  $(\tt{x}_i\beta, \tt{x}_{i+1}\beta$); $\beta$ is the same in each conjugate pair.  
We call such a path a \defo{chain} of length $m$, and define  \blue{$\firstg(\tt{x}_i\beta) = \tt{x}_1$}.
%
For each such chain in $\cycletree$, assign a permutation $d_1d_2\cdots d_m$  of  $\{1,2,\ldots, m\}$ in which no element appears in its original position (a derangement).  

Let $\alpha = \tt{a}_1\tt{a}_2\cdots \tt{a}_n$.  If $\alpha = \tt{x}_i\beta$ belongs to a conjugate pair that joins two nodes in a chain $\alpha_1, \alpha_2, \ldots, \alpha_m$ with corresponding derangement $d_1d_2\cdots d_m$, 
let $g(\alpha) = \tt{x}_{d_i}$.  
Then the following function $f_1$ is a successor rule for a corresponding universal cycle for $\mathbf{S}_{\cycletree}$ (based on the theory in~\cite{karyframework}):

\begin{center}
$f_{1}(\alpha) = \left\{ \begin{array}{ll}

 g(\alpha)  &\ \  \mbox{if $\alpha$ belongs to a conjugate pair  in $\mathbf{C}$;}\\
         {\tt{a}_1} \  &\ \  \mbox{otherwise.}\end{array} \right. $
         
\end{center}
When $k=2$, $f_1 = f_0$.

\begin{exam} \label{exam:chain2} \small
%
%
Continuing Example~\ref{exam:chain}, let $\alpha = 300$; it belongs to a conjugate pair. Note that $\alpha_1 = 001$, $\alpha_2 = 002$, and $\alpha_3 = 003$ form a chain of length $m=3$. If the derangement assigned to this chain is 231, then $f_1$ is the successor rule for the universal cycle 100200300.  If the derangement assigned to this chain is 312, then $f_1$ is the successor rule for the universal cycle 300200100.

%
%

\end{exam}

Perhaps the most natural derangements for the chains in $\cycletree$ are of the form $23\cdots m1$ and $m12\cdots (m{-}1)$. 
Specifically, let: 
\begin{itemize}
\item $\fup(\alpha)$ denote the function $f_1(\alpha)$ when all chain derangements have the form $23\cdots m1$, and 
\item $\fdown(\alpha)$ denote the function $f_1(\alpha)$ when all chain derangements have the form $m12\cdots (m{-}1)$.
\end{itemize}
These are precisely the successor rules that correspond to our upcoming concatenation tree results.
They are also the ones used in the generic successor rules stated in Theorem 2.8 and Theorem 2.9 from~\cite{karyframework}; they lead to the definition of natural successor rules for eight different DB sequences including the $k$-ary Granddaddy (lex smallest)~\cite{fkm2} and the $k$-ary Grandmama~\cite{grandma2}.

\subsection{Insights into concatenation trees} \label{sec:insight}


%
The sequence in Table~\ref{tab:db6} generated by $\pcr{1}$ starting with $0^n$ has an interesting property:  It corresponds to concatenating the \emph{aperiodic prefixes} of each node in the corresponding cycle-joining tree $\cycletree_1$ (illustrated in Figure~\ref{fig:big4}) as they are visited in post-order, where the children of a node are listed in lexicographic order. Notice also, that a post-order traversal visits the necklaces (nodes) as they appear in lexicographic order; this corresponds to the well-known Granddaddy necklace concatenation construction for DB sequences~\cite{fkm2}. 
Similarly, the sequence generated by the successor rule $\pcr{2}$  starting with $0^n$  corresponds to concatenating the aperiodic prefixes of each node in the corresponding cycle-joining tree $\cycletree_2$ as they are visited in pre-order, where the children of a node are listed in colex order.  This traversal visits the necklaces (nodes) as they appear in colex order, which is known as the Grandmama concatenation construction for DB sequences~\cite{grandma2}. 
Unfortunately, this \emph{magic} does not carry over to the trees $\cycletree_3$ and $\cycletree_4$, no matter how we order the children; the existing proofs for $\cycletree_1$ and $\cycletree_2$ offer no higher-level insights or pathways towards generalization. 

Our discovery to finding a concatenation construction for a given successor rule is to tweak the corresponding cycle-joining tree by:
(i)  determining the appropriate representative of each cycle,
(ii) determining an ordering of the children, and 
(iii) determining how the tree is traversed.
%
The resulting concatenation trees for $\cycletree_1, \cycletree_2, \cycletree_3$, and $\cycletree_4$, which are formally defined in  Section~\ref{sec:concat}, are illustrated in  Figure~\ref{fig:concatTree} for $n=6$.
The concatenation trees derived from $\cycletree_1$ and $\cycletree_2$ look very similar to the original cycle-joining trees. For the concatenation tree derived from $\cycletree_3$, the representatives are obtained by rotating the initial prefix of 0s of a necklace to the suffix; a post-order traversal produces the corresponding DB sequence in Table~\ref{tab:db6}. This traversal corresponds to visiting these  representatives in reverse lexicographic order that is equivalent
to a construction defined in~\cite{GS18}.  The concatenation tree derived from $\cycletree_4$ is non-trivial and proved to be the basis for discovering our more general result.  Each representative is determined from its parent, and the tree  differentiates ``left-children'' (blue dots)  from ``right-children'' (red dots).
A concatenation construction corresponding to $\pcr{4}$ is obtained by a somewhat unconventional traversal that recursively visits right-children, followed by the current node, followed by the left-children. 

\section{Bifurcated ordered trees}
\label{sec:bots}

Our new ``concatenation-tree'' approach to generating universal cycles and DB sequences relies on tree structures that mix together ordered trees and binary trees.
First we review basic tree concepts.  Then we introduce our notion of a bifurcated ordered tree together with a traversal called an RCL traversal.


An \defo{ordered tree} is a rooted tree in which the children of each node are given a total order.
For example, a node in an ordered tree with three children has a first child, a second child, and a third (last) child.
In contrast, a \defo{cardinal tree} is a rooted tree in which the children of each node occupy specific positions.
In particular, a \defo{$k$-ary tree} has $k$ positions for the children of each node.
For example, each child of a node in a $3$-ary tree is either a left-child, a middle child, or a right-child.


We consider a new type of tree that is both ordinal and cardinal; while ordered trees have one ``type'' of child, our trees will have two types of children.
We refer to such a tree as a \defo{bifurcated ordered tree} (\defo{BOT}), with the two types of children being \defo{left-children} and \defo{right-children}.
To illustrate bifurcated ordered trees, Figure \ref{fig:bots_3} provides all BOTs with $n=3$ nodes. 
\begin{figure}[h]
    \centering
        \includegraphics{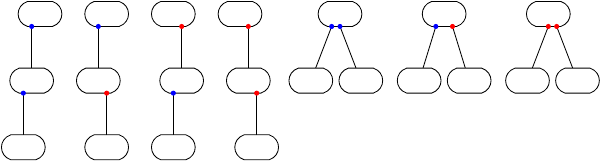}
        \caption{All eight bifurcated ordered trees (BOTs) with $n{=}3$ nodes.  Each left-child descends from a \blue{blue $\bullet$}, while each right-child descends from a \red{red $\bullet$}.       
        }
        \label{fig:bots_3}
\end{figure}
This type of ``ordinal-cardinal'' tree seems quite natural %
and it is very likely to have been used in previous academic investigations.  Nevertheless, the authors have not been able to find an exact match in the literature. In particular, $2$-tuplet trees use a different notion of a root, and correspond more closely to ordered forests of BOTs.  The number of BOTS with $n=1, 2, \ldots , 12$ nodes is given by:


\vspace{-0.1in}

$$1, 2, 7, 30, 143, 728, 3876, 21318, 120175, 690690, 4032015, 23841480.$$

\vspace{-0.05in}

\noindent
This listing corresponds to sequence A006013 in the Online Encyclopedia of Integer Sequences \cite{oeis1}. 



\subsection{Right-Current-Left (RCL) traversals}
\label{sec:bots_RCL}

The distinction between left-children and right-children in a BOT allows for a very natural notion of an \emph{in-order traversal}: 
visit the left-children from first to last, then the current node, then the right-children from first to last.
During our work with concatenation trees (see Section \ref{sec:concat}) it will be more natural to use a modified traversal, in which the right-children are visited before the left-children. 
Formally, we recursively define a \defo{Right-Current-Left (RCL) traversal} of a bifurcated ordered tree starting from the root as follows:

\vspace{-0.39in}

~~
\begin{wrapfigure}[8]{r}{0.33\textwidth}  
        \centering
       \resizebox{2.0in}{!}{\includegraphics{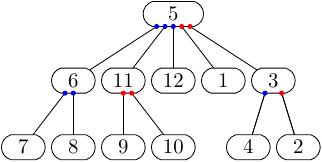}  }     
    \caption{A BOT with its $n{=}12$ nodes labeled as they appear in RCL order.} 
    \label{fig:bots_RCL}
\end{wrapfigure}
\begin{itemize}
    \item visit all right-children of the current node from first to last;
    \item visit the current node;
    \item visit all left-children of the current node from first to last.
\end{itemize}
Note that the resulting RCL order is not the same as a \emph{reverse in-order traversal} (i.e., an in-order traversal written in reverse), since the children of each type are visited in the usual order (i.e., first to last) rather than in reverse order (i.e., last to first).
An example of an RCL traversal is shown in Figure \ref{fig:bots_RCL}.


Define the following  relationships given a node $x$ in a BOT.
\begin{itemize} \small
    \item A \defo{right-descendant} of $x$ is a node obtained by traversing down zero or more right-children.
    \item A \defo{left-descendant} of $x$ is a node obtained by traversing down zero or more left-children.
    \item The \defo{rightmost left-descendant} of $x$ is the node obtained by repeatedly traversing down the last left-child as long as one exists.
    \item The \defo{leftmost right-descendant} of $x$ is the node obtained by repeatedly traversing down the first right-child as long as one exists.
\end{itemize}
Note that a node is its own leftmost right-descendent if it has no right-children.  Similarly, a node is its own rightmost left-descendent if it has no left-children.
The following remark details the cases for when two nodes from a BOT appear consecutively in RCL order; they are illustrated in Figure \ref{fig:nextRCL}.

\begin{remark} \label{rem:nextRCL}
If a bifurcated ordered tree has RCL traversal $\ldots, x, y, \ldots$, then one of the following three cases holds:
\begin{enumerate}[(a)]
    \item $x$ is an ancestor of $y$: $y$ is the leftmost right-descendant of $x$'s first left-child; 
    \item $x$ is a descendant of $y$: $x$ is the rightmost left-descendent of $y$'s last right-child; 
    \item $x$ and $y$ are descendants of a common ancestor $a$ (other than $x$ and $y$): $x$ is the rightmost left-descendant and $y$ is the leftmost right-descendant of consecutive left-children or right-children of $a$.
\end{enumerate} 
Moreover, if the traversal sequence is cyclic (i.e., $x$ is last in the ordering and $y$ is first), there are three additional cases:
\begin{enumerate}[(a),start=4]
    \item $x$ is an ancestor of $y$: $x$ is the root and $y$ is its leftmost right-descendant;
    \item $x$ is a descendant of $y$: $y$ is the root and $x$ is its rightmost left-descendant;
    \item $x$ and $y$ are descendants of a common ancestor $a$ (other than $x$ and $y$):  $x$ is the rightmost left-descendant of the root, and $y$ is the leftmost right-descendant of the root.
\end{enumerate}

\end{remark}
Figure~\ref{fig:nextRCL} illustrates the six cases from the above remark. The three cases provided for cyclic sequences are stated in a way to convince the reader that all options are considered; however, they can be collapsed to the single case (f) if we allow the common ancestor $a$ to be $x$ or $y$.  

\begin{figure}[h]
    \centering
    \resizebox{6in}{!}{\includegraphics{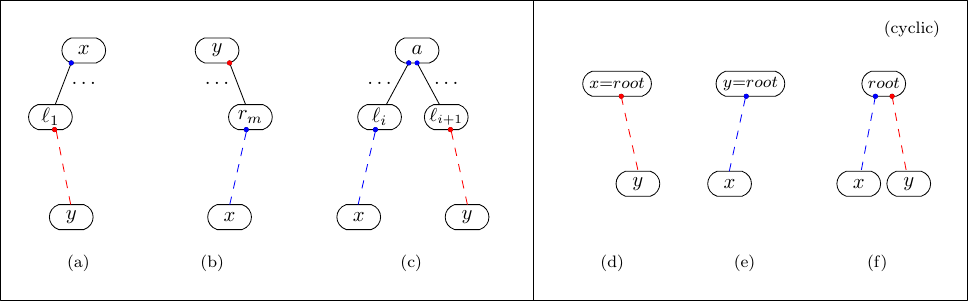}}
    \caption{ Illustrating the six cases outlined in Remark~\ref{rem:nextRCL} for when $y$ follows $x$ in an RCL traversal.  The final three cases hold when the traversal sequence is considered to be cyclic (i.e., $x$ comes last and $y$ comes first).   
    In these images, $\ell_i$ and $r_i$ refer to the $i$th left and right-child of their parent, respectively, and $r_m$ refers to the last right-child of its parent.
    Dashed lines indicate \red{leftmost right-descendants} (red) and \blue{rightmost left-descendants} (blue).
    }
       \label{fig:nextRCL}
\end{figure}

\section{Concatenation trees} 
\label{sec:concat}

Let $\cycletree$ be a PCR-based cycle-joining tree rooted at $r$ satisfying the Chain Property. In this section we describe how $\cycletree$  can be converted into a labeled BOT $\tree$ we call a \defo{concatenation tree}. The nodes and the parent-child relationship in $\tree$ are the same as in $\cycletree$; however, the labels (representatives) of the nodes may change.  The definitions of these labels are defined recursively along with a corresponding \defo{change index}, the unique index where a node's label differs from that of its parent.
The root of $\tree$ is $r$, and it is assigned an arbitrary change index $c$.\footnote{Though the change index of the root is arbitrary, its choice may impact the ``niceness'' of the upcoming RCL sequence.}  The label of a non-root node $\gamma$ depends on the label of its parent $\alpha = \tt{a}_1\tt{a}_2\cdots \tt{a}_n$, which
can be written as $\beta_1 \tt{x} \beta_2$ where
$(\tt{x}\beta_2\beta_1,\tt{y}\beta_2\beta_1)$  is the conjugate pair joining $\alpha$ and $\gamma$ in $\cycletree$.  If $\alpha$ is aperiodic, there is only one possible index $i$ for $\tt{x}$; however, if it is periodic, there will be multiple such indices possible.  
If $\alpha = (\tt{a}_1\cdots \tt{a}_p)^q$ has period $p$ with change index $c$ where $jp < c \leq jp+p$ for some integer $0 \leq j < n/p$, then we say the \defo{acceptable range} of $\alpha$ is $\{ jp{+}1, \ldots , jp{+}p\}$.
Note, if $\alpha$ is aperiodic, its acceptable range is $\{1, 2, \ldots , n\}$.  Now, 
$\alpha = \beta_1\tt{x} \beta_2$ can be written uniquely such that $\tt{x}$ is found at an index $i$ in $\alpha$'s acceptable range.  The label of $\gamma$ is defined to be $\beta_1\tt{y} \beta_2$ with change index $i$. 

\vspace{-0.1in}

\begin{exam} \small
Let  $x=001001001$ be the parent of $y=001002001$ in a PCR-based cycle-joining tree $\cycletree$ joined via the conjugate pair $(100100100, 200100100)$.  Let $\alpha$ and $\gamma$ denote the corresponding nodes in the concatenation tree $\tree$.
Suppose $\alpha = 1001001\underline{0}0$ (a rotation of $x$) with change index $8$.  Since $\alpha$ has period $p=3$, its acceptable range is $\{7,8,9\}$. Thus, $\beta_1 = 100100$, $\tt{x} = 1$, $\beta_2 = 00$, $\alpha = \beta_1 \tt{x} \beta_2$,  and
$\gamma = 100100\bblue{2}00$ (a rotation of $y$) with change index~$7$.
\end{exam}

\vspace{-0.1in}

To complete the definition of $\tree$, we must specify how the children of a node with change index $c$ are partitioned into ordered left-children and right-children: The left-children are those with change index less than $c$, and the right-children are those with change index greater than $c$.  Both are ordered by increasing change index.  A child with change index $c$ can be considered to be either a left-child or right-child.  We say $\tree$ is a 
\defo{left concatenation tree} if every node that has the same change index as its parent is considered to be a left-child; $\tree$ is a 
\defo{right concatenation tree} if every node that has the same change index as its parent is considered to be a right-child.  
 Let
\blue{$\convert(\cycletree,c,\mathit{left})$} denote the left concatenation tree
derived from $\cycletree$ and let 
\blue{$\convert(\cycletree,c,\mathit{right})$} denote the  right concatenation tree derived from $\cycletree$, where in each case the root is assigned change index $c$.  See Figure~\ref{fig:convert} for example concatenation trees, where the small gray box on top of each node indicates the node's change index.

\begin{figure}[ht]
     \centering
    \resizebox{6.5in}{!}{\includegraphics{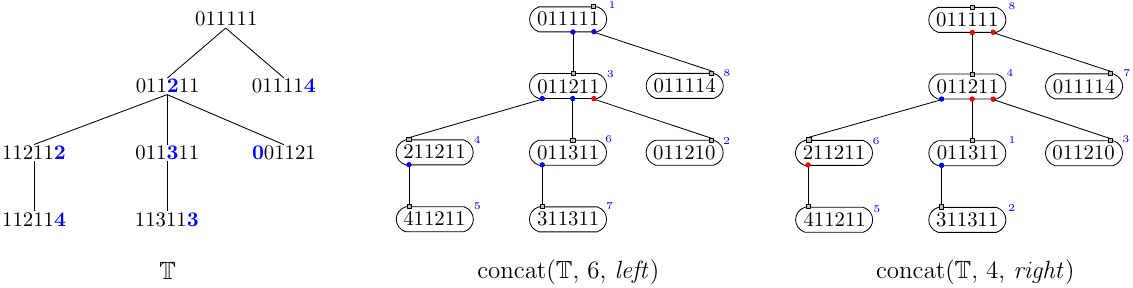}}

    \caption{Left and right concatenation trees  for a given cycle-joining tree $\cycletree$. The small blue numbers indicate the RCL order.}
    \label{fig:convert}
\end{figure}

%

 Let $\blue{\RCL(\bigtree{} ~) =  \ap(\alpha_1, \alpha_2, \ldots, \alpha_t)}$, where 
 $\alpha_1, \alpha_2, \ldots, \alpha_t$ is the sequence of nodes visited in an RCL traversal of the concatenation tree $\bigtree{}$. 
%
For example, if $\bigtree{}$ is the \emph{right} concatenation tree in Figure~\ref{fig:convert}, then:

\vspace{-0.25in}

\[ \RCL(\bigtree{} ~)  = 011311~311~011210~011211~411211~211~011114~011111. \]
\begin{wrapfigure}{r}{0.13\textwidth}
\centering
\vspace{-0.45in}
\resizebox{0.8in}{!}{\includegraphics{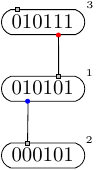}}
\end{wrapfigure}

\vspace{-0.25in}

\noindent
It is critical how we defined the acceptable range for periodic nodes, since our goal is to demonstrate that $\RCL(\bigtree{} ~)$ produces a universal cycle.  
For example, consider three necklace class representatives (a) 010111, (b) 010101, and (c) 000101 where $n=6$.  They can be joined by
flipping the last 0 in (b) and flipping the second 0 in (c); (a) is the parent of (b) and (b) is the parent of (c).  A BOT for this cycle-joining tree is shown on the right.
It is \emph{not} a concatenation tree since the change index for the bottom node is outside the acceptable range of its periodic parent.  
Observe that $\ap(010101, 000101, 010111) =  0100\underline{01\blue{0101}}\blue{01}11.$
Since the substring 010101 appears twice, it is not a universal cycle.

The concatenation trees for the four cycle-joining trees in Figure~\ref{fig:big4} are given in Figure~\ref{fig:concatTree}. 
The only concatenation tree with both left-children and right-children is the one corresponding to $\concat(\cycletree_4,6,\mathit{left})$.
In fact, it was the discovery of this tree that lead us to the introduction of BOTs and our definition of concatenation trees.  We are now ready to state our main result.
%
\begin{figure}[t]
     \centering
    \resizebox{6.5in}{!}{\includegraphics{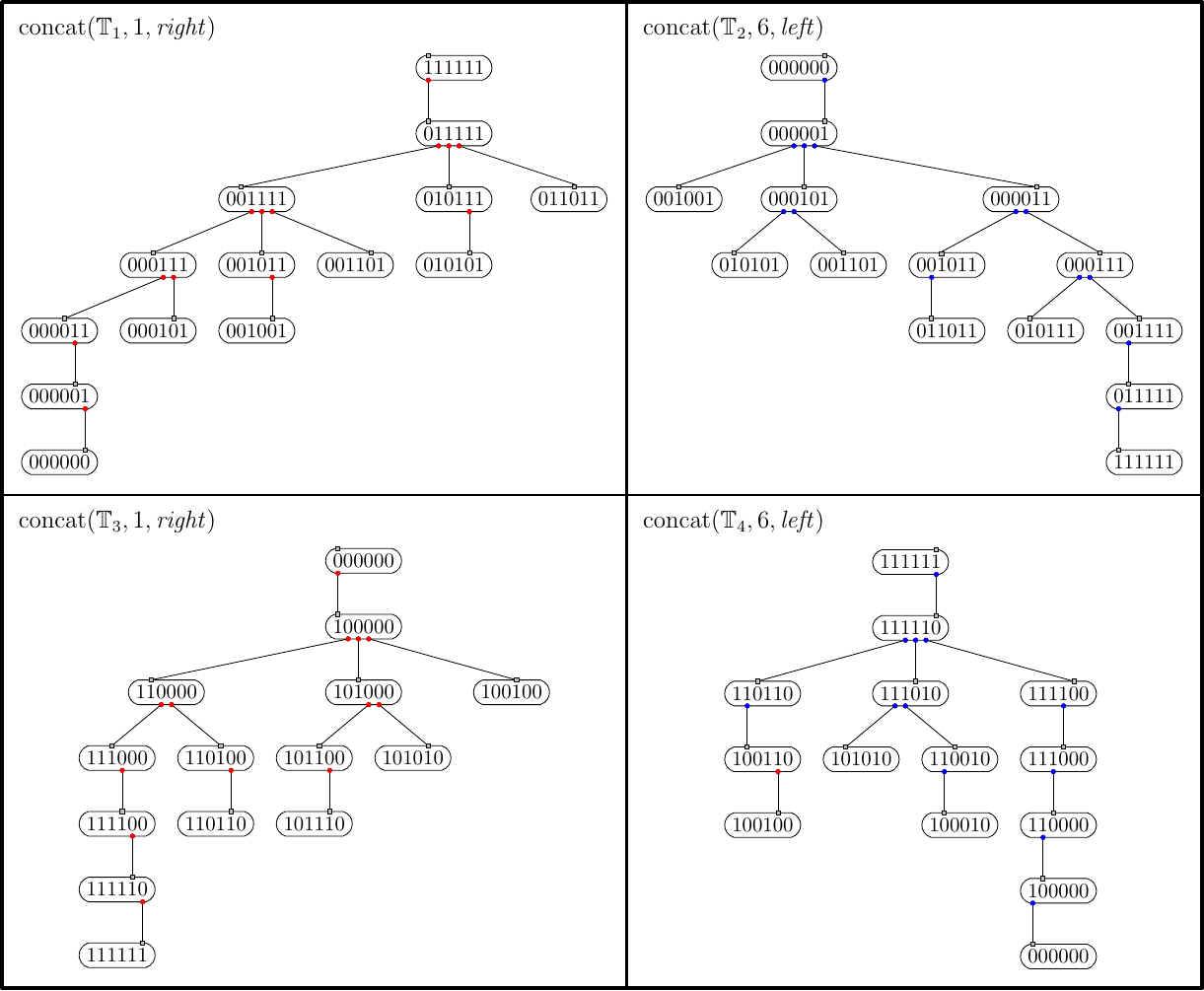}}

    \caption{Concatenation trees for $n=6$ based on $\cycletree_1$, $\cycletree_2$, $\cycletree_3$, $\cycletree_4$. 
    These bifurcated ordered trees (BOTs) provide additional structure to the unordered cycle-joining trees from Figure~\ref{fig:big4}.
    This structure provides the missing information for fully understanding the corresponding concatenation constructions.
    The gray box above each node indicates its change index.  
    }
    \label{fig:concatTree}
\end{figure}

\noindent
%


\begin{result}
\vspace{-0.1in}
\begin{theorem} \label{thm:main}
Let $\cycletree$ be a PCR-based cycle-joining tree satisfying the Chain Property.  Let $\tree_1 = \convert(\cycletree,c,\mathit{left})$ and let $\tree_2 = \convert(\cycletree,c,\mathit{right})$.
 Then
 \begin{itemize}
\item $\RCL(\tree_1~)$ is a universal cycle for $\mathbf{S}_{\cycletree}$ with successor rule $\fup$, and
\item $\RCL(\tree_2~)$ is a universal cycle for $\mathbf{S}_{\cycletree}$ with successor rule $\fdown$.
 \end{itemize}

\vspace{-0.1in}

\end{theorem}
\end{result}
\begin{proof}
Let $\tree$ represent either $\tree_1$ or $\tree_2$.  We specify whether $\tree$ is a left-concatenation tree $\tree_1$ or a right-concatenation tree $\tree_2$ only when necessary.
Let $\alpha_1, \alpha_2, \ldots, \alpha_t$ be the nodes of $\tree$ as they are visited in RCL order.
%
%
The proof of Theorem~\ref{thm:main} is by induction on $t$. 
In the base case case when $t=1$, the result is immediate; $\tree$ contains a single cycle and in each case the successor rule simplifies to $f(\tt{a}_1\tt{a}_2\cdots \tt{a}_n) = \tt{a}_1$. 
Suppose $t>1$.
%
%
 Let  $\alpha_j = \tt{a}_1\tt{a}_2\cdots \tt{a}_n$ denote an arbitrary leaf of $\bigtree{}$ with change index $c_j$. Let $\beta_1 = \tt{a}_1\cdots \tt{a}_{c_j-1}$, $\tt{y} = \tt{a}_{c_j}$, and $\beta_2 = \tt{a}_{c_j+1}\cdots \tt{a}_n$.
 Then $\alpha_j =  \beta_1 \tt{y} \beta_2$ and its parent is $\beta_1 \tt{y}' \beta_2$ for some $\tt{y}' \in \Sigma$; the
corresponding nodes in $\cycletree$ are joined via the conjugate pair $(\tt{y}\beta_1  \beta_2, \tt{y'}\beta_1\beta_2)$.
 %
\blue{If $\tree = \tree_1$, let $\tt{x} = \tt{y}'$;  if $\tree = \tree_2$, let $\tt{x} = \firstg(\tt{y'}\beta_1  \beta_2)$  with respect to $\cycletree$} (recalling the definition of $\firstg$ in Section~\ref{sec:succ-kary}).
Let $\tree'$ denote the concatenation tree obtained by removing $\alpha_j$ from $\bigtree{}$. Similarly, let $\cycletree'$ denote the cycle-joining tree $\cycletree$ with the leaf corresponding to $\alpha_j$ removed.
Let $U_1 = \ap(\alpha_{j+1}, \ldots, \alpha_t, \alpha_1, \ldots, \alpha_{j-1})$ denote a rotation of $\RCL(\tree')$.   
By induction, $U_1$ is a universal cycle for $\mathbf{S'} = \mathbf{S}_{\cycletree} - \rot{\alpha_j}$.
Let $U_2 = \ap(\alpha_j)$; it is a universal cycle for $\rot{\alpha_j}$.  %
Note that $U_1$ contains $\tt{x} \beta_2\beta_1$ and $U_2$ contains $\tt{y}\beta_2\beta_1$.  The following claim will be proved later in Section~\ref{sec:proof}.
\begin{claim} \label{claim:main}
$U_1$ (considered cyclically) has prefix $\beta_1$ and suffix $\tt{x}\beta_2$.
\end{claim}
Let $U'_1 = \cdots\tt{x}\beta_2\beta_1$ and let $U'_2 = \cdots \tt{y}\beta_2\beta_1$ be rotations of $U_1$ and $U_2$, respectively.
Then by Theorem~\ref{thm:concat} and Claim~\ref{claim:main}, $U_1$ and $U_2$ can be joined via the conjugate pair $(\tt{x}\beta_2\beta_1, \tt{y}\beta_2\beta_1)$ to produce universal cycle $U'_1U'_2$, which is a rotation of $U_1U_2$, for $\mathbf{S}_{\cycletree}$. Since $U_1U_2$ is a rotation of $U = \RCL(\bigtree{}~)$, the latter is also a universal cycle for $ \mathbf{S}_{\cycletree}$.
%


Clearly $\fup=\fdown$ with respect to the single PCR cycle $\rot{\alpha_j}$; both functions are successor rules for $U_2$. Suppose $\tree = \tree_1$. From the induction hypothesis, $\fup$ (with respect to $\cycletree'$) is a successor rule for $U_1$.  Since the two cycles $U_1$ and $U_2$ were joined via the conjugate pair $(\tt{x}\beta_2\beta_1, \tt{y}\beta_2\beta_1)$ to obtain $U$; the successors of only these two strings are altered. 
By the joining, the successor of $\tt{y}\beta_2\beta_1$ becomes the successor of $\tt{x}\beta_2\beta_1$ in $U_1$ which is precisely $\fup(\tt{y}\beta_2\beta_1)$ with respect to $\cycletree$.  The successor of $\tt{x}\beta_2\beta_1$ is $\tt{y}$, which is the same as $\fup(\tt{x}\beta_2\beta_1)$ with respect to $\cycletree$. Thus, $\fup$ (with respect to $\cycletree$) is a successor rule for $U$. A similar argument applies for $\tree = \tree_2$.  
\end{proof}

\begin{remark} \label{rem:unique}
Consider a cycle-joining tree $\cycletree$ where all chains in $\cycletree$ have length $m=2$.
Then $\cycletree$ induces a unique universal cycle with successor rule $\fup = \fdown$.  Furthermore, if $k=2$, $f_0 = \fup = \fdown$.
\end{remark}
\subsection{Algorithmic details and analysis} \label{sec:algo}

A concatenation tree can be traversed to produce a universal cycle in 
$O(1)$-amortized time per symbol; but, it requires exponential space to store the tree.  However, if the children of a given node $\alpha$ can be computed without knowledge of the entire tree, then we can apply Algorithm~\ref{algo:recursive} to traverse a  concatenation tree $\bigtree{}$  in a space-efficient manner. The initial call is \Call{RCL}{$\alpha$,$c$, $\ell$} where $\alpha=\tt{a}_1\tt{a}_2\cdots \tt{a}_n$ is the root node with change index $c$.  The variable $\ell$ is set to 1 for left concatenation trees; $\ell$ is set to 0 for right concatenation trees.
The crux of the algorithm is the function 
\Call{Child}{$\alpha,i$} which returns $\tt{x}$ if there exists $\tt{x} \in \Sigma$ such that $\tt{a}_1\cdots \tt{a}_{i-1}\tt{x}\tt{a}_{i+1}\cdots \tt{a}_n$ is a child of $\alpha$, or $-1$ otherwise.  Since the underlying cycle-joining tree satisfies the Chain Property, if such an $\tt{x}$ exists then it is unique.
In practice, the function must consider the acceptable range of $\alpha$.


\begin{algorithm}[ht]      
\caption{Traversing a  concatenation tree $\bigtree{}$ in RCL order rooted at $\alpha$ with change index $c$}  
\label{algo:recursive}      

\small
\begin{algorithmic} [1]                   
\Procedure{RCL}{$\alpha = \tt{a}_1\cdots \tt{a}_n$, $c$, $\ell$}

\For{$i\gets c+\ell$ {\bf to} $n$}  \ \ \ \ \  \  \ \ \blue{$\triangleright$ Visit right-children}  
    \State $x \gets \Call{Child}{\alpha,i}$
     \If{$\tt{x} \neq -1$} \ \  \Call{RCL}{$\tt{a}_1\cdots \tt{a}_{i-1}\tt{x}\tt{a}_{i+1}\cdots \tt{a}_n$, $i$, $\ell$} \EndIf
\EndFor
\State $p \gets $ period of $\alpha$
\State \Call{Print}{$\tt{a}_1\cdots \tt{a}_p$} 
\For{$i\gets 1$ {\bf to} $c-1+\ell$}  \ \   \blue{$\triangleright$ Visit left-children}  
    \State $\tt{x} \gets \Call{Child}{\alpha,i}$
     \If{$\tt{x} \neq -1$} \ \  \Call{RCL}{$\tt{a}_1\cdots \tt{a}_{i-1}\tt{x}\tt{a}_{i+1}\cdots \tt{a}_n$, $i$, $\ell$} \EndIf
\EndFor

\EndProcedure
\end{algorithmic}
\end{algorithm}

Let $H$ denote the height of $\bigtree{}$. Provided each call to \Call{Child}{$\alpha,i$} uses at most $O(n)$ space, the overall algorithm will require $O(n+H)$ space assuming $\alpha$ is passed by reference (or stored globally) and restored appropriately after each recursive call.   The running time of  Algorithm~\ref{algo:recursive} depends on how efficiently the function \Call{Child}{$\alpha,i$} can be computed for each index $i$.

\begin{theorem} \label{thm:algo}
   Let $\bigtree{}$ be a  concatenation rooted at $\alpha$ with change index $c$.   The sequence resulting from a  call to \Call{RCL}{$\alpha$, $c$, $\ell$} is generated in $O(1)$-amortized time per symbol if 
   (i) at each recursive step the work required by all calls to \Call{Child}{$\alpha,i$} is
   $O((t+1)n)$, where $t$ is the number of $\alpha$'s children, and 
   (ii) the number of nodes in $\bigtree{}$ that are periodic is less than some constant times the number of nodes that are aperiodic.
\end{theorem}
\begin{proof}
    The work done at each recursive step is $O(n)$ plus the cost associated to all calls to \Call{Child}{$\alpha,i$}.  If condition (i) is satisfied, then the work can be amortized over the $t$ children if $t\geq 1$, or onto the node itself if there are no children. Thus, each recursive node is the result of $O(n)$ work. By condition (ii), the total number of symbols output will be proportional to $n$ times the number of nodes. Thus, each symbol is output in $O(1)$-amortized time.
\end{proof}
%

\subsection{Properties of concatenation trees}
\label{sec:properties}


Let $\alpha_1, \alpha_2, \ldots, \alpha_t$ be the nodes of a concatenation tree $\tree$ as they are visited in RCL order.  Our proof of Claim~\ref{claim:main} relies on properties exhibited between successive nodes in an RCL traversal of $\tree$.   The operations of the indices are taken modulo $t$, i.e.,  $\alpha_0 = \alpha_t$ and $\alpha_{t+1} = \alpha_1$. Recall that $c_i$ denotes the change index of $\alpha_i$.
For the rest of this section, consider a node $\alpha_j = \tt{a}_1\tt{a}_2\cdots \tt{a}_n$, for some $1 \leq j \leq t$, with
change index $c_j$.  Let $\beta_1 = \tt{a}_1\tt{a}_2\cdots \tt{a}_{c_j-1}$, $\tt{y} = \tt{a}_{c_j}$ and let $\beta_2 = \tt{a}_{c_j+1}\cdots \tt{a}_{n}$; $\alpha_j = \beta_1 \tt{y} \beta_2$.

\begin{lemma} \label{lem:pre}
If $\alpha_j$ is not an ancestor of $\alpha_{j+1}$, then $\alpha_{j+1}$ has prefix $\beta_1$. 
\end{lemma}
%
\begin{proof}
%

\noindent
Let $x=\alpha_{j}$ and $y=\alpha_{j+1}$.  Following the notation from Figure~\ref{fig:nextRCL}, 
consider the four possible cases (b)(c)(e)(f) from Remark~\ref{rem:nextRCL}.  If $t=1$, the results are immediate. Suppose $t>1$.
  
%
\begin{enumerate}
%
\item[(b)] $r_m$ clearly has prefix $\beta_1$. Since the change index of $r_m$ is less than or equal to $c_j$, and $r_m$ only differs from its parent $y$ at its change index, $y$ must also have the prefix $\beta_1$. 

\item[(c)] $\ell_i$ clearly has prefix $\beta_1$. The change index of $\ell_i$ is strictly less than the change index of $\ell_{i+1}$ and the two nodes differ only at those two indices.
Thus, $\beta_1$ is a prefix of $\ell_{i+1}$.   Since $y$ can only differ from $\ell_{i+1}$ in indices between the change index of $\ell_{i+1}$ and $c_{j+1}$, it must also have the prefix $\beta_1$.


\item[(e)] Trivial.

\item[(f)]  All the nodes on the path from $x$ up to the root and down to $y$ must have change index greater than or equal to $c_j$.  Thus each node, including $y$ will have prefix $\beta_1$. 
\end{enumerate}
\end{proof}

\begin{lemma} \label{lem:leaf}
If $\alpha_j$ is not an ancestor of $\alpha_{j+1}$, and 
$\alpha_{j+1}$ has period $p<n$ with acceptable range $kp+1, \ldots , kp + p$, then $c_j \leq kp+p$.    
\end{lemma}
\begin{proof}
Let $x=\alpha_{j}$ and $y=\alpha_{j+1}$.  Following the notation from Figure~\ref{fig:nextRCL}, 
consider the four possible cases (b)(c)(e)(f) from Remark~\ref{rem:nextRCL}.  If $t=1$, the results are immediate. Suppose $t>1$.
\begin{enumerate}
\item[(b)] ~The change index for $r_m$ must be less than or equal to $kp+p$, and because $\alpha_j$ is a left descendant of $r_m$, $c_j$ must be less than or equal to the change index of $r_m$. Thus, $c_j \leq kp+p$. 
\item[(c)]  ~$c_j$ is less than or equal to the change index of $\ell_i$, which is less than the change index of $\ell_{i+1}$, which is less than or equal to $c_{j+1}$. Thus, $c_j < c_{j+1} \leq kp + p$. 
\item[(e)]  ~$\alpha_j$ is a left-descendant of $\alpha_{j+1}$ so clearly $c_j < c_{j+1} \leq kp + p$.
\item[(f)]  ~$c_j$ is less than or equal to the change index of the root, which is less than or equal to $c_{j+1}$. Thus, $c_j < c_{j+1} \leq kp + p$.
\end{enumerate}
\end{proof}

\noindent
If $\alpha_j$ is not an ancestor of  $\alpha_{j-1}$, then from Remark~\ref{rem:nextRCL}, $\alpha_j$ is not the root node and thus has a parent $\beta_1\tt{y'}\beta_2$ for some $\tt{y'} \in \Sigma$.
Recalling the definition of $\firstg$ in Section~\ref{sec:succ-kary} with respect to the underlying cycle-joining tree $\cycletree$, let $\tt{x} = \firstg(\tt{y'}\beta_1\beta_2)$.

\begin{lemma}  \label{lem:suf}
If $\alpha_j$ is not an ancestor of  $\alpha_{j-1}$, then
\begin{enumerate}
    \item if $\tree$ is a \blue{left} concatenation tree then $\alpha_{j-1}$ has suffix $\tt{y'}\beta_2$, and
    \item if $\tree$ is a \red{right} concatenation tree then $\alpha_{j-1}$ has suffix $\tt{x}\beta_2$.
\end{enumerate}
\end{lemma}
\begin{proof}
Let $x=\alpha_{j-1}$ and $y=\alpha_{j}$. Following notation from Figure~\ref{fig:nextRCL}, consider the four possible cases (a)(c)(d)(f) from Remark~\ref{rem:nextRCL}.
If $t=1$, the results are immediate. Suppose $t>1$.  
%
%
%
Suppose $\tree$ is a \blue{\emph{left}} concatenation tree.

\begin{enumerate}
\item[(a)] If $\ell_1 = y$, the result is immediate. Suppose $\ell_1 \neq y$. From the definition of $\tt{y'}$, $\ell_1$ has suffix $\tt{y'}\beta_2$ and change index strictly less than $c_j$. Since $\ell_1$  differs from its parent $x$ only at its change index, $x$ must also have  suffix $\tt{y'}\beta_2$.  
%
%
\item[(c)]  If $\ell_{i+1} = y$, then it is already established that its parent $a$ has suffix $\tt{y'}\beta_2$.  Otherwise, $\ell_{i+1}$ has suffix $\tt{y'}\beta_2$ and change index less than $c_j$, which means that $a$ again has suffix $\tt{y'}\beta_2$. Since the change index of $\ell_i$ is less than the change index of $\ell_{i+1}$, clearly $x$ also has suffix $\tt{y'}\beta_2$.
\item[(d)] Follows since $\tree$ is a left concatenation tree.
%
%
\item[(f)] Let $\alpha_r$ be the root of $\bigtree{}$.
Clearly, $\alpha_r$ has suffix $\tt{y'}\beta_2$ and $c_r < c_j$. Thus,
$x$ also will have suffix $\tt{y'}\beta_2$.
\end{enumerate}
Suppose $\tree$ is a \red{\emph{right}} concatenation tree. Let $\tt{x} =  \firstg(\tt{y'}\beta_1  \beta_2)$.  This implies that all nodes on the path from $\beta_1\tt{x}\beta_2$ to $y=\alpha_j$ have change index $c_j$ and the change index of $\beta_1\tt{x}\beta_2$ is not equal to $c_j$.
\begin{enumerate}
\item[(a)]
If $\ell_1 = y$, then the change index of $\ell_1$ is strictly less than the change index of $x$ and the result follows as $\tt{x} = \tt{y'}$.  
Suppose $\ell_1 \neq y$. If the change index of $\ell_1$ is strictly less than $c_j$, then by the definition of $\tt{x}$, $\ell_1$ has suffix $\tt{x}\beta_2$. Thus, clearly $x$ also has suffix $\tt{x}\beta_2$.
Otherwise, the change index of $\ell_1$ must be equal to $c_j$, and since it is a left-child of $x$, the change index of $x$ is not equal to $c_j$.  Thus, by the definition of $\tt{x}$, $x$ will be precisely $\beta_1 \tt{x} \beta_2$.
%
%
\item[(c)] Recall this covers two cases where the children of $a$ can be either be both left-children or both right-children. In either case, the change index of $a$ can not be the same as the change index for $\ell{i+1}$.
Thus, following the same argument from (a), the node $a$ will have suffix $\tt{x}\beta_2$. Since the change index of $\ell_i$ is less than the change index of $\ell_{i+1}$, clearly $x$ also has suffix $\tt{x}\beta_2$.
\item[(d)] Follows since $\tree$ is a right concatenation tree.
%
%
\item[(f)] Let $\alpha_r$ be the root of $\tree$.
Clearly, $\alpha_r$ has suffix $\tt{x}\beta_2$ and $c_r \leq c_j$. 
Since all left descendants of the root will have change index strictly
less than $c_r$, it follows that $x$ also will have suffix $\tt{x}\beta_2$.
\end{enumerate}
\end{proof}

\begin{lemma} \label{lem:leaf2}
If $\alpha_j$ is not an ancestor of $\alpha_{j-1}$, and
$\alpha_{j-1}$ has period $p<n$ with acceptable range $kp+1, \ldots , kp + p$, then $c_j > kp$.   
\end{lemma}
\begin{proof}
Let $x=\alpha_{j-1}$ and $y=\alpha_{j}$. Following notation from Figure~\ref{fig:nextRCL}, consider the four possible cases (a)(c)(d)(f) from Remark~\ref{rem:nextRCL}.
If $t=1$, the results are immediate. Suppose $t>1$.  
\begin{enumerate}
\item[(a)]~By the acceptable range, the change index for $\ell_1$ must be greater than $kp$. Because $\alpha_j$ is a right descendant of $\ell_1$, $c_j$ must be greater than or equal to the change index of $\ell_1$. Thus, $c_j > kp$.
\item[(c)]~$c_{j-1}$ is less than or equal to the change index of $\ell_i$, which is less than the change index of $\ell_{i+1}$, which is less than or equal to $c_j$. Thus, $kp < c_{j-1} < c_j$. 
\item[(d)]~$\alpha_j$ is a right-descendant of $\alpha_{j-1}$ so clearly $kp < c_{j-1} < c_{j}$.
\item[(f)]~$c_{j-1}$ is less than or equal to the change index of the root, which is less than $c_j$. Thus, $kp < c_{j-1} < c_j$.
\end{enumerate}
\end{proof}


\begin{lemma} \label{lem:pre_period}
If $\alpha_j$ is periodic with period $p$ and acceptable range $kp+1, \ldots , kp + p$, then $\ap(\alpha_j)^{k}$ is a prefix of $\alpha_{j+1}$. 
\end{lemma}
\begin{proof}
If $\alpha_{j}$ is not an ancestor of $\alpha_{j+1}$, the inequality $kp < c_j$ and Lemma~\ref{lem:pre} together imply $\ap(\alpha_j)^k$ is a prefix of $\alpha_{j+1}$.
It remains to consider cases (a) and (d) from Remark~\ref{rem:nextRCL} where $x = \alpha_{j}$ is an ancestor of $y = \alpha_{j+1}$.  For case (a), $y$ is the leftmost right-descendent of $x$'s first left-child $\ell_1$.  Since $x$ is periodic, the change index of $\ell_1$ is in $\alpha_j$'s acceptable range; it is greater than $kp$. $y$ is a right descendant of $\ell_1$ and thus $c_{j+1} > kp$, which means $y$ differs from $\ell_1$ only in indices greater than $kp$.
For (d) clearly $y$ differs only in indices greater than or equal to $c_j$, which means $c_{j+1} > kp$.
Thus, for each case, $\ap(\alpha_j)^{k}$ is a prefix of $\alpha_{j+1}$.  
\end{proof}

\begin{lemma} \label{lem:pre_period_strong}
If $\alpha_j$ is periodic with period $p$ and acceptable range $kp+1, \ldots,  kp + p$,  then $\ap(\alpha_j)^{k+1}$ is a prefix of $\ap(\alpha_j, \ldots, \alpha_t,  \alpha_1, \ldots, \alpha_{j-1})$, which is a rotation of $\RCL(\bigtree{}~)$, considered cyclically.
\end{lemma}
\begin{proof}  Note that $|\ap(\alpha_j)^{k+1}| \leq n$.
The proof is by induction on the number of nodes $t$.
If $t=1$, the result is trivial.
Suppose the claim holds for any tree with less than $t > 1$ nodes. Let $\tree$ have $t$ nodes and let $\alpha_j$ be a leaf node of $\tree$.  If there are no periodic nodes, we are done.  Otherwise, we first consider $\alpha_j$, then all other periodic nodes in $\tree$.

Suppose $\alpha_j$ is periodic
with period $p$ and acceptable range $kp + 1, \ldots, kp + p$. From Lemma~\ref{lem:pre_period}, $\ap(\alpha_j)^k$ is a prefix of $\alpha_{j+1}$. If $\alpha_{j+1}$ is aperiodic, then we are done. 
Suppose, then, that $\alpha_{j+1}$ is periodic with period $p'$ and acceptable range $k'p' + 1, \ldots, k'p' + p'$. Let $\bigtree{}'$ be the tree resulting from $\bigtree{}$ when $\alpha_j$ is removed. It follows from (i) that $kp < c_j \leq k'p' + p'$, which implies $\ap(\alpha_j)^k$ is a prefix of $\ap(\alpha_{j+1})^{k'+1}$. Additionally, since $\tree'$ has less than $t$ nodes and $\alpha_{j+1}$ is periodic, $\ap(\alpha_{j+1})^{k'+1}$ is a prefix of $\ap(\alpha_{j+1}, \ldots, \alpha_t, \alpha_1, \ldots, \alpha_{j-1})$ by our inductive assumption. Therefore, $\ap(\alpha_j)^{k+1}$ is a prefix of $\ap(\alpha_j,\alpha_{j+1}, \ldots \alpha_t, \alpha_1, \ldots, \alpha_{j-1})$.  

Now consider $\alpha_{j-1}$.  If it is aperiodic, then by induction, the claim clearly holds for all periodic nodes in $\bigtree{}'$.  Thus, assume $\alpha_{j-1}$ is periodic.
By showing that $\ap(\alpha_{j-1},\alpha_j), \ldots, \alpha_t, \alpha_1, \ldots, \alpha_{j-2})$ has the desired prefix, then repeating the same arguments will prove the claim holds for every other periodic node in $\tree'$.
Let $\alpha_{j-1}$ have period $p''$ and acceptable range $k''p'' + 1, \ldots, k''p'' + p''$. 
If $\alpha_j$ is aperiodic, Lemma~\ref{lem:pre_period} implies that $\ap(\alpha_{j-1})^{k''}$ is a prefix of $\alpha_j = \ap(\alpha_j)$ and thus the claim holds for $\alpha_{j-1}$. 
If $\alpha_j$ is periodic with period $p$ and acceptable range $kp + 1, \ldots, kp + p$, we already demonstrated that 
$\ap(\alpha_j)^{k+1}$ is a prefix of $\ap(\alpha_j, \ldots, \alpha_t, \alpha_1, \ldots, alpha_{j-1})$.
From Lemma~\ref{lem:pre_period}, $\ap(\alpha_{j-1})^{k''}$ is a prefix of $\alpha_j$.  Note that (i) and its proof  handles cases (b)(c)(e)(f) from Remark~\ref{rem:nextRCL} implying that $c_{j-1} < kp+p$ for these cases.  Since $\alpha_{j-1}$ is not necessarily a leaf, we must also consider (a) and (d). In both cases, clearly $k''p'' < c_j$.  Either way, $k''p''< kp+p$,  which means
$\ap(\alpha_{j-1})^{k''}$ is a prefix of $\ap(\alpha_j)^{k+1}$.
Thus, $\ap(\alpha_{j-1})^{k''+1}$ is a prefix of $\ap(\alpha_{j-1}, \alpha_j, \ldots, \alpha_t, \alpha_1, \ldots, \alpha_{j-2})$.
\end{proof}

\begin{lemma} \label{lem:suf_period}
If $\alpha_j$ is periodic with period $p$ and acceptable range $kp+1, \ldots , kp + p$, then $\ap(\alpha_j)^{n/p-k-1}$ is a suffix of $\alpha_{j-1}$.   
\end{lemma}

\begin{proof}
If $\alpha_{j}$ is not an ancestor of $\alpha_{j-1}$, the inequality $c_j \leq kp + p$ and Lemma~\ref{lem:suf} together imply $\ap(\alpha_j)^{n/p-k-1}$ is a suffix of $\alpha_{j-1}$. It remains to consider cases (b) and (e) from Remark~\ref{rem:nextRCL} where $y = \alpha_{j}$ is an ancestor of $x = \alpha_{j-1}$.  For case (b), $x$ is the rightmost left-descendent of $y$'s last right-child $r_m$. Since $y$ is periodic, the change index of $r_m$ is in $\alpha_j$'s acceptable range; it is less than or equal to $kp+p$. $x$ is a left descendant of $r_m$ and thus $c_{j-1} \leq kp + p$, which means $x$ differs from $r_m$ only in indices less than or equal to $kp + p$. For (e) clearly $x$ differs only in indices less than or equal to $c_j$, which means $c_{j-1} \leq kp+p$.
Thus, for each case, $\ap(\alpha_j)^{n/p-k-1}$ is a suffix of $\alpha_{j-1}$.  
\end{proof}

\begin{lemma} \label{lem:suf_period_strong}
If $\alpha_j$ is periodic with period $p$ and acceptable range $kp+1, \ldots, kp + p$, then $\ap(\alpha_j)^{n/p-k}$  is a suffix of $\ap(\alpha_{j+1}, \ldots, \alpha_t, \alpha_1, \ldots, \alpha_j)$,
which is a rotation of $\RCL(\bigtree{}~)$, considered cyclically.
\end{lemma}

\begin{proof}  Let $q=n/p$. Note that $|ap(\alpha_j)^{q-k}| \leq n$.
The proof is by induction on $t$.  If $t=1$, the result is trivial. Suppose the claim holds for any tree with less than $t > 1$ nodes. Let $\tree$ have $t$ nodes and let $\alpha_j$ be a leaf node of $\tree$.
If there are no periodic nodes, we are done.  Otherwise, we first consider $\alpha_j$, then all other periodic nodes in $\tree$.

Suppose $\alpha_j$ is periodic with period $p$ and acceptable range $kp+1, \ldots, kp + p$. From Lemma~\ref{lem:suf_period}, $\ap(\alpha_j)^{q-k-1}$ is a suffix of $\alpha_{j-1}$. If $\alpha_{j-1}$ is aperiodic, then we are done. 
Suppose, then, that $\alpha_{j-1}$ is periodic with period $p'$ and acceptable range $k'p'+1, \ldots, k'p' + p'$. Let $\tree'$ be the tree resulting from $\tree$ when $\alpha_j$ is removed. It follows from (i) that $k'p' < c_j \leq kp + p$, or $n - kp - p < n - k'p'$, which implies $\ap(\alpha_j)^{q-k-1}$ is a suffix of $\ap(\alpha_{j-1})^{q'-k'}$, where $q' = n / p'$. Additionally, since $\tree'$ has less than $t$ nodes and $\alpha_{j-1}$ is periodic, $\ap(\alpha_{j-1})^{q'-k'}$ is a suffix of $\ap(\alpha_{j+1}, \ldots, \alpha_t, \alpha_1, \ldots, \alpha_{j-1})$ by our inductive assumption. Therefore, $\ap(\alpha_{j})^{q-k}$ is a suffix of $\ap(\alpha_{j+1}, \ldots, \alpha_t, \alpha_1, \ldots, \alpha_{j-1}, \alpha_{j})$.

Now consider $\alpha_{j+1}$. If it is aperiodic, then by induction the claim clearly holds for all periodic nodes in $\tree'$.  Thus,
assume $\alpha_{j+1}$ is periodic.
By showing $\ap(\alpha_{j+2}, \ldots, \alpha_{t}, \alpha_{1}, \ldots, alpha_{j}, \alpha_{j+1})$ has the desired suffix, then  repeating the same arguments will prove the claim holds for every other periodic node in $\tree'$.
Let $\alpha_{j+1}$ have period $p''$ and acceptable range $k''p'' + 1, \ldots, k''p'' + p''$. 
If $\alpha_j$ is aperiodic, Lemma~\ref{lem:suf_period} implies that $\ap(\alpha_{j+1})^{q''-k''-1}$ is a suffix of $\alpha_j = \ap(\alpha_j)$ and thus the claim holds for $\alpha_{j+1}$. 
If $\alpha_j$ is periodic with period $p$ and acceptable range $kp + 1, \ldots, kp + p$, we already demonstrated that $\ap(\alpha_j)^{q-k}$ 
is a suffix of $\ap(\alpha_{j+1}, \ldots, \alpha_t, \alpha_1, \ldots, \alpha_{j})$.
From Lemma~\ref{lem:suf_period}, $\ap(\alpha_{j+1})^{q''-k''-1}$ is a suffix of $\alpha_j$. Note that (i) and its proof handles cases (a)(c)(d)(f) from Remark~\ref{rem:nextRCL} implying that $c_{j+1} > kp$ for these cases.  Since $\alpha_{j+1}$ is not necessarily a leaf, we must also consider (b) and (e).  In both cases, clearly $c_j \leq k''p'' + p''$. Either way, $kp  < k''p'' + p''$, which means
$\ap(\alpha_{j+1})^{q''-k''-1}$ is a suffix of $\ap(\alpha_j)^{q-p}$.  Thus,
$\ap(\alpha_{j+1})^{q''-k''}$ is a suffix of $\ap(\alpha_{j+2}, \ldots, \alpha_t, \alpha_1, \ldots, \alpha_{j},\alpha_{j+1})$.
\end{proof}

\subsubsection{Proof of Claim~\ref{claim:main}}
\label{sec:proof}

Recall that $U_1 = \ap(\alpha_{j+1}, \ldots, \alpha_t, \alpha_1, \ldots, \alpha_{j-1})$ and  $\alpha_j = \beta_1 \tt{y} \beta_2$. Recall that $\alpha_j$ is assumed to be a leaf with parent $\beta_1 \tt{y}' \beta_2$. Also, if $\tree = \tree_1$, then $\tt{x} = \tt{y}'$;  if $\tree = \tree_2$,  $\tt{x} = \firstg(\tt{y'}\beta_1  \beta_2)$.
Thus, from Lemma~\ref{lem:pre}, $\alpha_{j+1}$ has prefix $\beta_1$ and
from Lemma~\ref{lem:suf}, $\alpha_{j-1}$ has suffix $\tt{x}\beta_2$.
If $\alpha_{j-1}$ and $\alpha_{j+1}$ are aperiodic, then $U_1$ has prefix $\beta_1$ and suffix $\tt{x}\beta_2$ as required. If $t=2$, then we are also done since $U_1$ is considered cyclically.
It remains to consider the cases where  $\alpha_{j-1}$ or $\alpha_{j+1}$ is periodic and $t>2$. These cases apply the ``acceptable range''. 

Suppose $\alpha_{j+1}$ has period $p<n$ and acceptable range $kp+1, \ldots, kp + p$. Since $\alpha_j$ is a leaf, $c_j \leq kp+p$ by Lemma~\ref{lem:leaf}. Thus, $\beta_1$ is a prefix of $\ap(\alpha_{j+1})^{k+1}$ since $\beta_1$ is a prefix of $\alpha_{j+1}$ from Lemma~\ref{lem:pre}.  From 
Lemma~\ref{lem:pre_period_strong}, $\ap(\alpha_{j+1})^{k+1}$ is a prefix of $U_1$. Thus, $\beta_1$ is a prefix of $U_1$.


Suppose $\alpha_{j-1}$ has period $p<n$ and acceptable range $kp+1, \ldots , kp + p$. Since $\alpha_j$ is a leaf, 
$c_j > kp$ by Lemma~\ref{lem:leaf2}.
Thus, $\tt{x}\beta_2$ is a suffix of $\ap(\alpha_{j-1})^{n/p-k}$ since  $\tt{x}\beta_2$ is a suffix of $\alpha_{j-1}$ from Lemma~\ref{lem:suf}.
 From 
Lemma~\ref{lem:suf_period_strong}, $\ap(\alpha_{j-1})^{n/p-k}$ is a suffix of $U_1$. Thus, $\tt{x}\beta_2$ is a suffix of $U_1$.


\section{Applications} \label{sec:app}

In this section we highlight how our concatenation tree framework can be applied to a variety of interesting objects including permutations, weak orders, orientable sequences, and  DB sequences with related universal cycles.  
For each object, we  define a PCR-based cycle-joining tree $\cycletree$ that satisfies the Chain Property, where each node is a necklace (the lex smallest representative).  Then we apply the concatenation tree framework and Algorithm~\ref{algo:recursive} to construct the corresponding universal cycles in $O(1)$-amortized time per symbol. 
\subsection{De Bruijn sequences and relatives} \label{sec:dbseq}

Recall that $\pcr{1}$, $\pcr{2}$, $\pcr{3}$, and $\pcr{4}$ are stated generally for subtrees of the corresponding cycle-joining trees $\cycletree_1, \cycletree_2,\cycletree_3, \cycletree_4$; they focus on binary strings, and thus satisfy the Chain Property. 
Though we focus on the binary case, these trees can be generalized to larger alphabets following the theory in~\cite{karyframework}.  For instance, the parent rule used to create $\cycletree_1$ can be generalized to ``increment the last non-$(k{-}1)$'' where the alphabet is $\Sigma = \{0,1,\ldots , k-1\}$. 
\begin{theorem} \label{cor:pcr}
Let $T_1, T_2, T_3, T_4$ be subtrees of $\cycletree_1, \cycletree_2,\cycletree_3, \cycletree_4$, respectively.
For any $1 \leq c \leq n$ and $\ell \in \{\mathit{left}, \mathit{right}\}$: 
\begin{itemize}
\item  $U_1 = \RCL(\convert(T_1,c,\ell))$ is a universal cycle for $\mathbf{S}_{T_1}$ with successor rule $\pcr{1}$.
\item  $U_2 = \RCL(\convert(T_2,c,\ell))$ is a universal cycle  for $\mathbf{S}_{T_2}$ with successor rule $\pcr{2}$.
\item  $U_3 = \RCL(\convert(T_3,c,\ell))$ is a universal cycle  for $\mathbf{S}_{T_3}$ with successor rule $\pcr{3}$.
\item  $U_4 = \RCL(\convert(T_4,c,\ell))$ is a universal cycle  for $\mathbf{S}_{T_4}$ with successor rule $\pcr{4}$.
\end{itemize}
\end{theorem}
\begin{proof}
The results follow immediately from Remark~\ref{rem:unique} and Theorem~\ref{thm:main}.
\end{proof}
Interesting subtrees applied to the above theorem include nodes with:
(i) a lower bound on weight ($T_1$ and $T_4$),
(ii) an upper bound on weight ($T_2$ and $T_3$),
(iii) forbidden $0^s$ substring ($T_1$ and $T_4$),
(iv) forbidden $1^s$ substring ($T_2$ and $T_3$).
Universal cycles for strings with bounded weight (based on $T_1$ and $T_2$)~\cite{walcom,dbrange,min-weight} and universal cycles with forbidden $0^s$ (based on $T_1$)~\cite{GS18,generalized-greedy} have been previously studied; the latter has found recent application in quantum key distribution schemes~\cite{cheeRLL}. Theorem~\ref{cor:pcr} generalizes these result and provides a connection between the concatenation constructions and corresponding successor rules.

If the subtrees in Theorem~\ref{cor:pcr} are $\cycletree_1, \cycletree_2, \cycletree_3$, and $\cycletree_4$, respectively, the universal cycles are DB sequences.  Specifically, let  
\begin{itemize}
    \item $\tree_{1} = \convert(\cycletree_1,1,\mathit{right})$,
    \item $\tree_{2} = \convert(\cycletree_2,n,\mathit{left})$,
    \item $\tree_{3} = \convert(\cycletree_3,1,\mathit{right})$, and
    \item $\tree_{4} = \convert(\cycletree_4,n,\mathit{left})$.
\end{itemize}

\noindent
Recall that the Granddaddy DB sequence can be constructed by concatenating the aperiodic prefixes of necklaces as they appear in lexicographic order; it is known to have the successor rule $\pcr{1}$, and the sequence can be generated in $O(1)$-amortized time  per bit.

\begin{corollary}
  $\RCL(\tree_{1})$ is the Granddaddy DB sequence with successor rule $\pcr{1}$.
\end{corollary}
\begin{proof}
$\tree_{1}$ is based on the ``last 0'' cycle-joining tree rooted at $1^n$,  the change index of the root is $1$, and the tree is right-concatenation tree. Thus, the representatives of each node are necklaces (flipping the last 0 of a necklace yields a necklace), where the change index of each child is after the last 0.  This means
each child is a right-child that is lexicographically smaller than its parent, and the children are ordered lexicographically left to right. 
Therefore, RCL order is a post-order traversal of $\tree_{1}$ that visits the necklaces $\mathbf{N}_2(n)$ in lexicographic order.  Thus, $\RCL(\tree_{1})$ is the Granddaddy DB sequence, and by  Theorem~\ref{cor:pcr} it has successor rule $\pcr{1}$.
\end{proof}

The Grandmama DB sequence can be constructed by concatenating the aperiodic prefixes of necklaces as they appear in co-lexicographic order; it is known to have the successor rule $\pcr{2}$, and the sequence can be generated in $O(1)$ amortized time  per bit.
%

\begin{corollary}
 $\RCL(\tree_{2})$ is the Grandmama DB sequence with successor rule $\pcr{2}$. 
\end{corollary}
\begin{proof}

 $\tree_{2}$ is based on the ``first 1'' cycle-joining tree rooted at $0^n$,  the change index of the root is $n$, and the tree is left-concatenation tree. Thus, the representatives of each node are necklaces (flipping the first 1 of a necklace yields a necklace), where the change index of each child is before the first 1.  This means each child is a left-child that is lexicographically larger than its parent and the children are ordered co-lexicographically. Therefore, RCL order is a pre-order traversal of $\tree_{2}$ that visits the necklaces $\mathbf{N}_2(n)$ in co-lexicographic order. Thus, $\RCL(\tree_{2})$ is the Grandmama DB sequence, and by  Theorem~\ref{cor:pcr} it has successor rule $\pcr{2}$.
 \end{proof}


\noindent
Though the concatenation-tree framework is not necessary to obtain a more efficient DB sequence construction for the Granddaddy and Grandmama, there is a significant improvement in the simplicity of verifying both the correctness and the equivalence of the concatenation and successor rule constructions. 

We now  answer an unproved claim about the correspondence between the DB sequence generated by $\pcr{3}$ and a concatenation construction from~\cite{GS18}.  In particular,
let the representative of each necklace class be the necklace with the initial prefix of 0s rotated to the suffix, so each representative (except $0^n$) begins with 1.
The construction from~\cite{GS18}  concatenates the aperiodic prefixes of these representatives as they appear in reverse lexicographic order; here we name it the \emph{Granny} DB sequence.   As an example, see the sequence generated by $\pcr{3}$ in Table~\ref{tab:db6}. 
This sequence can also be generated in $O(1)$-amortized time  per bit~\cite{wong}.

\begin{corollary}
  $\RCL(\tree_{3})$ is the Granny DB sequence with successor rule $\pcr{3}$.
\end{corollary}
\begin{proof}
$\tree_{3}$ is based on the ``last 1'' cycle-joining tree rooted at $0^n$,  the change index of the root is $1$, and the tree is right-concatenation tree. Thus, it is not difficult to observe recursively that the representative of each non-root node is a necklace with the initial prefix of 0s is rotated to the suffix, so it begins with a 1.
Furthermore, the change index of each child is in the rotated suffix of 0s.
This means each child is a right-child that is lexicographically smaller than its parent and the children are ordered in reverse lexicographic order. Therefore, RCL order is a post-order traversal of $\tree_{3}$ that visits the representatives in reverse-lexicographic order.  Thus, $\RCL(\tree_{3})$ is the Granny DB sequence, and by  Theorem~\ref{cor:pcr} it has successor rule $\pcr{3}$.
\end{proof}

%

Recall that the question of whether or not there existed a ``simple'' concatenation construction for $\pcr{4}$ was the motivating question that lead to the discovery of concatenation trees. Unfortunately, it appears as the though the resulting RCL traversal is not so simple; each
node representative appears to require recursive information about its parent (and hence the tree structure).  Here, we name the sequence constructed by $\RCL(\tree_{4})$ the \emph{Grandpa} DB sequence. Experimental evidence indicates an algorithm that runs in $O(1)$-amortized time per bit may exist using the concatenation tree framework; however, the analysis appears non-trivial.

\begin{corollary}
  $\RCL(\tree_{4})$ is the Grandpa DB sequence with successor rule $\pcr{4}$.
\end{corollary}

\subsection{Universal cycles for shorthand permutations} \label{sec:perm}

A \defo{shorthand} permutation is a length $n{-}1$ prefix of some permutation $\tt{p}_1\tt{p}_2\cdots \tt{p}_n$. Let $\mathbf{SP}(n)$ denote the set of shorthand permutations of order $n$.  For example, $\mathbf{SP}(3) = \{12,13, 21, 23, 31, 32\}$.
Note that $\mathbf{SP}(n)$ is closed under rotation.  The  necklace classes of $\mathbf{SP}(n)$ can be joined into a PCR-based cycle-joining tree via the following parent rule~\cite{karyframework},  
where each cycle representative is a necklace. 

\begin{result} \small
\noindent
{\bf Parent rule for cycle-joining the necklaces in $\mathbf{SP}(n)$:}  Let $r$ denote the root  $12\cdots (n-1)$. Let $\sigma$ denote a non-root node where $\tt{z}$ is the missing symbol.  If $\tt{z} = n$, let $j>1$ denote the smallest index such that $\tt{p}_j < \tt{p}_{j-1}$,
otherwise let $j$ denote the index of $\tt{z}+1$.  Then 
\[ \parent(\sigma) = \tt{p}_1\cdots \tt{p}_{j-1}\blue{\tt{z}}\tt{p}_{j+1}\cdots \tt{p}_{n-1}. \]

\vspace{-0.15in}
 \end{result}

\noindent
Let $\cycletree_{perm}$ be the cycle-joining tree derived from the above parent rule; it satisfies the Chain Property. Observe that a node $\sigma = \tt{p}_1\tt{p}_2\cdots \tt{p}_{n-1}$ in $\cycletree$ will have at most two children. In particular, 
if $\tt{z}$ is the missing symbol, $\tt{p}_1\cdots \tt{p}_j
\tt{z} \tt{p}_{j+1} \cdots \tt{p}_{n-1}$ is a child of $\sigma$ if and only if 
(i) $\tt{p}_j = \tt{z}{-}1$, or (ii) $\tt{p}_j = n$, $\tt{p}_1\cdots \tt{p}_{j-1}$ is increasing, and $\tt{z} < \tt{p}_{j-1}$. 
Figure~\ref{fig:perm} illustrates $\cycletree_{perm}$ and $\bigtree{perm} = \convert(\cycletree_{perm},n{-}1,\mathit{left})$ for $n=4$.  Let $U_{perm} = \RCL(\bigtree{perm})$, when $n$ is understood. Then, for $n=4$: 
\[ U_{perm} = 123~124~132~143~243~142~134~234.\]  
A unique successor rule $f_{perm}$ for the universal cycle derived from $\cycletree_{perm}$ can be computed in $O(n)$ time~\cite{karyframework}.  

\begin{figure}[ht]
    \centering
        \resizebox{3.1in}{!}{\includegraphics{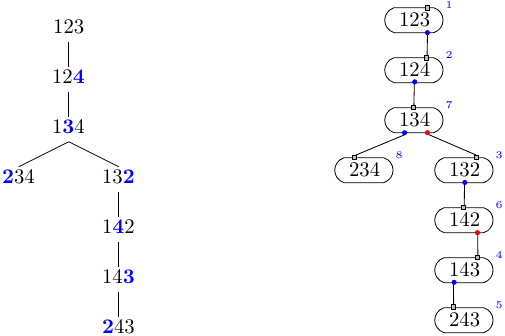}}
        \caption{A cycle joining tree $\cycletree_{perm}$ of shorthand permutation necklaces for $n=4$, and the corresponding concatenation tree $\bigtree{perm}$  illustrating the RCL order.   
        }
        \label{fig:perm}
\end{figure}
\begin{theorem}
$U_{perm}$ is a universal cycle for $\mathbf{SP}(n)$ with successor rule $f_{perm}$.  Moreover, $U_{perm}$ can be constructed in $O(1)$-amortized time per symbol using $O(n^2)$ space.
\end{theorem}
\begin{proof}
Since each chain in $\cycletree_{perm}$ has length at most $2$, $f_{perm}$ is unique and $f_{perm} = \fup = \fdown$ (see Remark~\ref{rem:unique}).
Thus, by Theorem~\ref{thm:main}, $U_{perm}$ is a universal cycle for $\mathbf{SP}(n)$ with successor rule $f_{perm}$.  As noted earlier, each node $\sigma$ in $\cycletree_{perm}$ has at most two nodes; they can easily be determined in linear time with a single scan of $\sigma$ and the values can be stored using a constant amount of space.  Thus, by Theorem~\ref{thm:algo}, $U_{perm}$ can be constructed in $O(1)$-amortized time per symbol.  The space required by the algorithm is proportional to the height of $\cycletree_{perm}$; each recursive call requires a constant amount of space.  
Consider a node $\sigma = \tt{p}_1\tt{p}_2\cdots \tt{p}_{n-1}$ in $\cycletree_{perm}$, where $j$ is the smallest index such that $\tt{p}_j < \tt{p}_{j-1}$. By applying at most $n$ applications of the parent rule, $\sigma$ has an ancestor whose length-$j$ prefix is increasing and the $j$-th symbol is $n$.  Thus, after at most $n^2$ applications of the parent rule, $\sigma$ has an ancestor $\sigma'$ that is increasing and ends with $n$.  After at most $n$ applications of the parent rule, $\sigma'$ will have the root $r$ as an ancestor.  Thus, the height of $\cycletree_{perm}$ is $O(n^2)$.
\end{proof}

\normalsize

Efficient concatenation constructions of universal cycles for shorthand permutations are  known~\cite{shorthand2,shorthand}; however (i) there is no clear connection between their construction and corresponding successor rule and (ii) they do not have underlying PCR-based cycle-joining trees.  

\subsection{Universal cycles for weak orders} \label{sec:weak}

Recall that $\mathbf{W}(n)$ denotes the set of weak orders of order $n$; it is closed under rotation. Weak orders of order $n$ are counted by the ordered Bell or Fubini numbers (OEIS A000670~\cite{oeis1}). 
The first construction of a universal cycle for $\mathbf{W}(n)$ defined the upcoming PCR-based cycle-joining tree, where the cycle-representatives (nodes) are the lex-smallest representatives (necklaces)~\cite{weakorder}.
Let $\omega = \tt{w}_1\tt{w}_2\cdots \tt{w}_n$ denote a string in $\mathbf{W}(n)$.
Let $n_{\omega}(i)$ denote the number of occurrences of the symbol $i$ in $\omega$. Let $\mathbf{W_1}(n)$ denote the set of all weak orders of order $n$ with no repeating symbols other than perhaps $1$.

\begin{result}
\noindent  \small
{\bf Parent rule for cycle-joining the necklaces in $\mathbf{W}(n)$:}  Let $r$ denote the root  $1^n$. Let $\omega = \tt{w}_1\tt{w}_2\cdots \tt{w}_n$ denote a non-root node.  If $\omega \in \mathbf{W_1}(n)$, let $j$ denote the index of the symbol $n_{\omega}(1)+1$ and let $\tt{x}=1$; otherwise let $j$ be the largest index containing a repeated (non-$1$) symbol  and let $\tt{x} = \tt{w}_j + n_{\omega}(\tt{w}_j)-1$.
Then 
\[ \parent(\omega) = \tt{w}_1\cdots \tt{w}_{j-1}\blue{\tt{x}}\tt{w}_{j+1}\cdots \tt{w}_{n}. \]

\vspace{-0.15in}

 \end{result}

\noindent
Let $\cycletree_{weak}$ be the cycle-joining tree derived from the above parent rule; it clearly satisfies the Chain Property. Figure~\ref{fig:weak} illustrates both $\cycletree_{weak}$ and $\bigtree{weak} = \convert(\cycletree,n,\mathit{left})$ for $n=4$.
A successor rule $f_{weak}$ for the universal cycle based on $\cycletree_{weak}$ can be computed in $O(n)$ time~\cite{weakorder}.

Our goal is to apply Theorem~\ref{thm:algo} to construct an universal cycle for weak orders in $O(1)$-amortized time.
Consider a node $\omega =  \tt{w}_1\tt{w}_2\cdots \tt{w}_n$ in $\cycletree_{weak}$. 
Let $\tt{c}_1\tt{c}_2\cdots \tt{c}_n$ denote a sequence such that  $\tt{c}_i = \tt{x}$ if there exists an $\tt{x}$ such that the necklace of $\tt{w}_1\cdots \tt{w}_{i-1}\blue{\tt{x}}\tt{w}_{i+1}\cdots \tt{w}_n$ is a child of $\omega$ in $\cycletree$, or $-1$ otherwise.  Note that $\tt{x}$ is unique since $\cycletree_{weak}$ satisfies the Chain Property.  

\begin{lemma} \label{lem:noneck}
If $\alpha = \tt{a}_1\tt{a}_2\cdots \tt{a}_n$ is a necklace then $\beta = \tt{a}_j \cdots \tt{a}_n \tt{a}_1\cdots \tt{a}_{i-1}\tt{x}\tt{a}_{i+1}\cdots \tt{a}_{j-1}$ is not a necklace for any 
$1 \leq i < j \leq n$ where   $\tt{x}  < \tt{a}_i$.
\end{lemma}
\begin{proof}
Suppose $1 \leq i < j \leq n$.  Since $\alpha$ is a necklace, $\tt{a}_j \cdots \tt{a}_n \tt{a}_1\cdots \tt{a}_{i} \geq \tt{a}_1\cdots \tt{a}_{n-j+i+1}$. If $\beta$ is a necklace then the length $i$ prefix of $\beta$ must be less than or equal to $\tt{a}_1\cdots \tt{a}_{i-1}\tt{x}$ which is less than $ \tt{a}_1\cdots \tt{a}_{i}$ since $\tt{x}  < \tt{a}_i$. Contradiction.
\end{proof}

\begin{lemma} \label{lem:weak}
Let $\omega =  \tt{w}_1\tt{w}_2\cdots \tt{w}_n$ be a necklace in $\mathbf{W}(n) \setminus \mathbf{W}_1(n)$.  Given $1 \leq i \leq n$, if $\tt{w}_i > 1$ let $\tt{y}$ be the largest symbol smaller than $\tt{w}_i$ in $\omega$; otherwise, let $\tt{y} = 1$.
If $\tt{w}_1 \cdots \tt{w}_{i-1}\tt{y}\tt{w}_{i+1}\cdots \tt{w}_n$ is not a necklace then $\tt{c}_i = -1$.
\end{lemma}
\begin{proof}
 From the parent rule, a 1 is never changed to $\tt{x}$.  Thus, if $\tt{y} = 1$ then $\tt{c}_i = -1$.  Suppose $\tt{y} > 1$.  Let $\omega'$ be the necklace of $\tt{w}_1 \cdots \tt{w}_{i-1}\tt{y}\tt{w}_{i+1}\cdots \tt{w}_n$. 
Since $\omega'$ clearly begins with a 1, from Lemma~\ref{lem:noneck}, $\omega'$ starts 
with $\tt{w_j}$ for some $j < i$.
Since $\omega$ and $\omega'$ are both necklaces, $\tt{w}_1\cdots \tt{w}_{i-j} = \tt{w}_j\cdots \tt{w}_{i-1}$.
Consider all occurrences of $\tt{x} \in \omega$.  If $\tt{x}$ does not 
appear before index $j$, then since $\tt{w}_1\cdots \tt{w}_{i-j} = \tt{w}_j\cdots \tt{w}_{i-1}$, it must appear at an index after $i$.  Either way, the $\tt{y}$ slotted into position $i$ of $\omega'$ is not the right most repeated symbol in the corresponding necklace representative $\omega''$, and thus the parent of $\omega''$ is not $\omega$. Thus, from the parent rule, $\tt{c}_i = -1$.
\end{proof}

\noindent
The sequence $\tt{c}_1\tt{c}_2\cdots \tt{c}_n$ can be determined for a necklace $\omega$  by considering the following two cases from the parent rule. 
\begin{enumerate}
    \item Suppose $\omega \in \mathbf{W}_1$.  For each $1 \leq j \leq n$, if $\tt{w}_j = 1$ then $\tt{c}_j = n_{\omega}(1)$; otherwise, $\tt{c}_j = -1$
    \item Suppose $\omega \notin \mathbf{W}_1$. 
    Let $i$ denote the largest index such that $\tt{w}_i > 1$ and $n_{\omega}(\tt{w_i}) > 1$; all symbols in $\tt{w}_{i+1} \cdots \tt{w}_n$ are unique within $\omega$.  
    Consider $1 \leq j \leq n$.  If $\tt{w}_j = 1$, then clearly by the parent rule $\tt{c}_i = -1$. Otherwise, let $\tt{x}$ denote the largest symbol in $\omega$ less than $\tt{w}_j$ and let $\omega' = \tt{w}_1 \cdots \tt{w}_{i-1}\tt{x}\tt{w}_{i+1}\cdots \tt{w}_n$.
    %
    If $1 \leq j < i$, then $\tt{x}$ is not the rightmost (non-1) repeated symbol in $\omega'$, and by the parent rule $\tt{c}_j = -1$.
    If $i=j$, then by the parent rule, then $\tt{c}_i = -1$, since $\tt{w}_i$ is a repeated symbol.
    Suppose $i+1 \leq j \leq n$.  If $\tt{x} = 1$ or $\omega'$ is not a necklace, then by the parent rule and Lemma~\ref{lem:weak}, respectively, $\tt{c}_j = -1$.  Otherwise, suppose $\tt{x} > 1$ and $\omega'$ is a necklace. 
    Then $\tt{c}_j  = \tt{x}$ if $\tt{x}$ does not appear in $\tt{w}_{j+1}\cdots \tt{w}_n$; otherwise,
    $\tt{x}$ is not the rightmost (non-1) repeated symbol in $\omega'$, and thus $\tt{c}_j = -1$.
\end{enumerate} 
Applying the two cases above, Algorithm~\ref{alg:weak} computes the values $\tt{c}_1\tt{c}_2\cdots \tt{c}_n$ for $\omega$.

\begin{algorithm}[h] 
\caption{ Computing $\tt{c}_1\tt{c}_2\cdots \tt{c}_n$ for given node $\omega = \tt{w}_1\tt{w}_2\cdots \tt{w}_n$.}
\label{alg:weak}  

\small
\begin{algorithmic}[1]


    \State $\tt{c}_1\tt{c}_2\cdots \tt{c}_n \gets (-1)^n$

   \If{$\omega \in \mathbf{W}_1(n)$}    ~~~\blue{$\triangleright$ Case 1}
        \For{$i$ {\bf from} $1$ {\bf to} $n$}
        \If{$\tt{w}_i = 1$} $\tt{c}_i \gets n_{\omega}(1)$ \EndIf
        \EndFor
  
   \EndIf
   
   \If{$\omega \notin \mathbf{W}_1(n)$}       ~~~\blue{$\triangleright$ Case 2}
       \State $v_1v_2\cdots v_n \gets 0^n$ ~~~\blue{$\triangleright$ $v_i$ is set to 1 if symbol $i$ is visited in the following loop}
    \For{$i$ {\bf from} $n$ {\bf down to} $1$}  
        \If{$\tt{w}_i > 1 $ {\bf and} $n_\omega(\tt{w}_i) > 1$ }\ {\bf break}
        \Else
                \State $\tt{x} \gets$ the largest symbol in $\omega$ less than $\tt{w}_i$, or 0 if $\tt{w}_i=1$
                

             \If{$\tt{x} > 1$ {\bf and} \Call{IsNecklace}{$\tt{w}_1 \cdots \tt{w}_{i-1}\tt{x}\tt{w}_{i+1}\cdots \tt{w}_n$} {\bf and} $v_{\tt{x}} = 0$ } $\tt{c}_i \gets \tt{x}$ 
            \EndIf
        \EndIf

         \State $v_{\tt{w}_i}  \gets 1$
    \EndFor
\EndIf


\end{algorithmic}
\end{algorithm}


\noindent
If $n\leq 8$, there is at most one call to \Call{IsNecklace}{} on line 11 of Algorithm~\ref{alg:weak} that returns false, for a given input $\omega$.  For $n=9$, there are 10 strings for which the function returns false twice.  One of these strings is $\omega = 147914\rred{8}1\bblue{6}$, where the highlighted symbols correspond to the indices $i$ where such a call returns false.  The corresponding strings tested by the algorithm (in the order tested) are $14791481\bblue{4}$ and $147914\rred{6}16$.  Neither are necklaces.  After the second test, the following lemma demonstrates that the next non-1 symbol $\tt{w}_i$ considered by the {\bf for} loop (line 7) must be a repeated symbol in $\omega$.

\begin{lemma} \label{lem:2neck}
    There are at most two calls to \Call{IsNecklace}{$\tt{w}_1 \cdots \tt{w}_{i-1}\tt{x}\tt{w}_{i+1}\cdots \tt{w}_n$}  in Algorithm~\ref{alg:weak} (line 11) where $\tt{w}_1 \cdots \tt{w}_{i-1}\tt{x}\tt{w}_{i+1}\cdots \tt{w}_n$ is not a necklace.
\end{lemma}
\begin{proof}
    We trace Algorithm~\ref{alg:weak} and the {\bf for} loop (line 7) noting that $\omega = \tt{w}_1\tt{w}_2\cdots \tt{w}_n$ is a weak order necklace representative not in $\mathbf{W}_1(n)$.  We demonstrate that if \Call{IsNecklace}{} returns false twice, then the
    next iteration of the loop where $\tt{w}_i > 1$ must have
    $n_{\omega}(\tt{w}_i) > 1$.  Thus, the loop breaks (line 8) and there are no further calls to \Call{IsNecklace}{}.

    Consider two iterations of the {\bf for} loop (line 7) where the iterator has value $i$ and $j$, respectively, such that both iterations make a call
    to \Call{IsNecklace}{} (line 11) that returns false.  Furthermore, assume $j<i$ are the two largest values such that this is the case.
    Let $\alpha_i = \tt{w}_1 \cdots \tt{w}_{i-1}\tt{x}_i\tt{w}_{i+1}\cdots \tt{w}_n$ noting that
    $\tt{w}_i > \tt{x}_i > 1$ (lines 10 and 11) and $n_{\omega}(\tt{w}_i) = 1$ (line 8).
    Since $\alpha_i$ is not a necklace, by Lemma~\ref{lem:noneck}, there exists some largest index $1 \leq t_i \leq i$ such that the rotation of $\alpha$ starting from index $t_i$ is a necklace.  Thus, $\tt{w}_{t_i}\cdots \tt{w}_{i-1} \leq \tt{w}_1\cdots \tt{w}_{i-t_i}$. Since $\omega$ is not a necklace,  $\tt{w}_{t_i}\cdots \tt{w}_{i-1} \geq \tt{w}_1\cdots \tt{w}_{i-{t_i}}$.  Thus, $\tt{w}_{t_i}\cdots \tt{w}_{i-1} = \tt{w}_1\cdots \tt{w}_{i-{t_i}}$ (*).  Define $\alpha_j$, $\tt{x}_j$, and $t_j$ similarly, which means $\tt{w}_j  > \tt{x}_j > 1$ and $n_{\omega}(\tt{w}_j) = 1$.
    If $t_i \leq j <  i$, then  $n_{\omega}(\tt{w}_j) > 1$ by (*), a contradiction.  
    Thus, $t_j \leq j < t_i$.  Since $n_{\omega}(\tt{w}_j) = 1$,  $\tt{w}_j \neq \tt{w}_i$. 
    
\end{proof}

Let $U_{weak} = \RCL(\bigtree{weak})$.
\begin{theorem}
$U_{weak}$ is a universal cycle for $\mathbf{W}(n)$ with successor rule $f_{weak}$.  Moreover, $U_{weak}$ can be constructed in $O(1)$-amortized time per symbol using $O(n^2)$ space.
\end{theorem}
\begin{proof}
Based on the parent rule for $\cycletree_{weak}$, every chain in $\bigtree{weak}$ has length $m=2$.  From Remark~\ref{rem:unique}, $f_{weak} = \fup = \fdown$, and Theorem~\ref{thm:main} implies that $U_{weak}$ is a universal cycle for $\mathbf{W}(n)$ with successor rule $f_{weak}$.  

To generate $U_{weak}$, Algorithm~\ref{algo:recursive}
can apply Algorithm~\ref{alg:weak} to determine the children of a node by computing and storing $\tt{c}_1\tt{c}_2\cdots \tt{c}_n$; each recursive call requires $O(n)$ space.  The height of $\cycletree_{weak}$ is $O(n)$;  the path from any leaf to the root requires less than $n$ applications of $\parent(\omega)$ to break all the non-1 ties, and then less than $n$ applications of $\parent(\omega)$ to convert all the non-1s to 1s.  Thus, the construction requires $O(n^2)$ space.  The time to generate $U_{weak}$ depends on how efficiently we can compute $\tt{c}_1\tt{c}_2\cdots \tt{c}_n$ in Algorithm~\ref{alg:weak}.  The
values $n_\omega(\tt{w}_i)$ can be computed in $O(n)$ time cumulatively.  Applying these values, the cumulative time to compute $\tt{x}$ at line 10 is also $O(n)$, since each $\tt{w}_i > 1$ must be unique by line 8.  Thus, excluding the calls to \Call{IsNecklace}{}, the algorithm runs in $O(n)$ time.
Each call to \Call{IsNecklace}{} (which requires $O(n)$ time~\cite{Booth}) that returns true leads to a child (some $c_i > -1$).  From Lemma~\ref{lem:2neck}, there at most two calls to \Call{IsNecklace}{} that return false.  
Thus, the total work by Algorithm~\ref{alg:weak} is $O((t+1)n)$, where $t$ is the number of children of the input node $\omega$.
Hence by Theorem~\ref{thm:algo}, $U_{weak}$ can be constructed in $O(1)$-amortized time.
\end{proof}

\noindent

\subsection{Orientable sequences} \label{sec:orient}

An \defo{orientable sequence} is a universal cycle for a set $\mathbf{S} \subseteq \{0,1\}^n$ such that if $\tt{a}_1\tt{a}_2\cdots \tt{a}_n \in \mathbf{S}$, then its reverse
$\tt{a}_n\cdots \tt{a}_2\tt{a}_1 \notin \mathbf{S}$.  Thus, $\mathbf{S}$ does not contain palindromes. Orientable sequences do not exist for $n < 5$, and a maximal length orientable sequence for $n=5$ is $001011$. Somewhat surprisingly, the maximal length of binary orientable sequences are known only for $n=5,6,7$.  Orientable sequences were introduced in~\cite{Dai} with applications related to robotic position sensing.  They established upper and lower bounds for their maximal length; the lower bound is based on the \emph{existence} of a PCR-based cycle-joining tree, though no construction of such a tree was provided. 

A \defo{bracelet class} is an equivalence class of strings under rotation and reversal; its lexicographically smallest representative is a \defo{bracelet}.  A bracelet $\alpha$ is \defo{symmetric}
if $[\alpha] = [\alpha^R]$; otherwise it is \defo{asymmetric}.  
Let $\A(n)$ denote the set of all binary asymmetric bracelets of length $n$.  For example, 
$\A(8) = \{ 
00001011,
00010011,
00010111,
00101011,
00101111,
00110111\}.$
If $\A(n) = \{\alpha_1, \alpha_2, \ldots, \alpha_t\}$, let $\mathbf{O}(n) =   [\alpha_1]  \cup [\alpha_2] \cup \cdots \cup [ \alpha_t ]$.

Motivated by the work in~\cite{Dai}, a cycle-joining tree $\cycletree_{orient}$ for $\A(n)$ was discovered leading to the construction of an orientable sequence with asymptotically optimal length~\cite{orient}.  The parent rule combines three of the four ``simple'' parent rules defined earlier for PCR-based cycle joining trees; it applies the following functions where $\alpha = \tt{a}_1\tt{a}_2\cdots \tt{a}_n \in \A(n)$: 
\begin{itemize}
    \item $\firstone(\alpha)$ be the necklace $\tt{a}_1\cdots\tt{a}_{i-1}\bblue{0}\tt{a}_{i+1}\cdots \tt{a}_n$, where $i$ is the index of the first $1$ in $\alpha$;
    \item $\lastone(\alpha)$ be the necklace in $[\tt{a}_1\tt{a}_2 \cdots \tt{a}_{n-1}\bblue{0}]$;  
    \item $\lastzero(\alpha)$ be the necklace $\tt{a}_1\cdots\tt{a}_{j-1}\rred{1}\tt{a}_{j+1}\cdots \tt{a}_n$, where $j$ is the index of the last $0$ in $\alpha$.
\end{itemize}

\begin{result}
\noindent  \small
{\bf Parent rule for cycle-joining $\A(n)$:} 
Let $r$ denote the root $0^{n-4}1011$.  Let $\alpha$ denote a non-root node in $\A(n)$.
Then 
\[ \parent(\alpha) =  \mbox{ the first asymmetric bracelet in the list:  }
 ~\firstone(\alpha),~  \lastone(\alpha),~ \lastzero(\alpha). \]

\vspace{-0.1in}

 \end{result}

\noindent
Let $\cycletree_{orient}$ be the cycle-joining tree derived from the above parent rule. 
Figure~\ref{fig:orient} illustrates $\cycletree_{orient}$ together with  $\bigtree{orient} = \convert(\cycletree_{orient},n,\mathit{right})$ for $n=8$; an RCL traversal of $\bigtree{orient}$ produces the following orientable sequence of length 48: 
\[ 
00001011~
11001011~
10011011~
10001001~
10001011~
00101011. \]

%

\begin{figure}[h]
    \centering
\resizebox{5.0in}{!}{\includegraphics{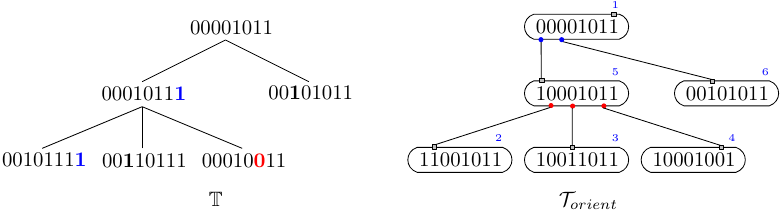}}
    \caption{ A cycle-joining tree for $\A(8)$  based on the parent rule $\parent(\alpha)$ followed by a corresponding right concatenation tree $\bigtree{orient}$ that illustrates the RCL order. }
       \label{fig:orient}
\end{figure}

A successor rule $f_{orient}$ based on $\cycletree_{orient}$ constructs the corresponding orientable sequence in $O(n)$-time per symbol~\cite{orient}. 
The following result is proved by applying Theorem~\ref{thm:main} and Theorem~\ref{thm:algo},
where $U_{orient} = \RCL(\bigtree{orient})$.

\begin{theorem}[\cite{orient}]
$U_{orient}$ is a universal cycle for $\mathbf{O}(n)$ with successor rule $f_{orient}$.  Moreover, $U_{orient}$ can be constructed in $O(1)$-amortized time per symbol using $O(n^2)$ space.
\end{theorem}
\noindent
\normalsize

\bibliographystyle{acm.bst}
\bibliography{refs}

\newpage

\appendix

\end{document}